\documentclass[11pt]{amsart}
\usepackage{geometry}
\usepackage{amsmath, amssymb}
 \usepackage{amsmath,amscd}
\usepackage{amsfonts}
\usepackage{mathrsfs}
\usepackage[arrow,matrix,curve,cmtip,ps]{xy}
\usepackage{setspace}

\usepackage{pb-diagram} 
\usepackage{pb-xy}
\usepackage{graphicx}

\usepackage{amsthm}

\usepackage{color}

\usepackage{xcolor}

\allowdisplaybreaks

\newtheorem*{rep@theorem}{\rep@title}
\newcommand{\newreptheorem}[2]{%
\newenvironment{rep#1}[1]{%
 \def\rep@title{#2 \ref{##1}}%
 \begin{rep@theorem}}%
 {\end{rep@theorem}}}
\makeatother

\newtheorem{theorem}{Theorem}
\newreptheorem{theorem}{Theorem}

\newtheorem{proposition}[theorem]{Proposition}
\newreptheorem{proposition}{Proposition}
\newtheorem{corollary}[theorem]{Corollary}

\newtheorem*{theorem*}{Theorem}
\newtheorem*{proposition*}{Proposition}
\newtheorem*{questions*}{Questions}
\newtheorem{definition}[theorem]{Definition}

\newtheorem{example}[theorem]{Example}
\newtheorem{question}{Question}
\newtheorem{remark}[theorem]{Remark}

\newcommand{\lnk}{\operatorname{lk}}
\newcommand{\sgn}{\operatorname{sign}}
\newcommand{\C}{\mathcal{C}}
\newcommand{\N}{\mathbb{N}}
\newcommand{\Z}{\mathbb{Z}}
\renewcommand{\c}{c}
\newcommand{\bdry}{\partial}
\newcommand{\Arf}{\operatorname{Arf}}

\newcommand{\WH}{\operatorname{WH}}

\begin{document}

\title[Links admitting Homeomorphic C-complexes]{When do links admit homeomorphic $C$-complexes?}

\author{Grant Roth}
\address{Department of Mathematics, University of Wisconsin-Eau Claire, Hibbard Humanities Hall 508,  Eau Claire WI 54702-4004}
\email{rothgm@uwec.edu}

\author{Christopher William Davis}
\address{Department of Mathematics, University of Wisconsin-Eau Claire, Hibbard Humanities Hall 508,  Eau Claire WI 54702-4004}
\email{daviscw@uwec.edu}

\date{\today}

\subjclass[2010]{}


\begin{abstract} 
Any two knots admit orientation preserving homeomorphic Seifert surfaces, as can be seen by stabilizing.  There is a generalization of a Seifert surface to the setting of links called a C-complex.  In this paper, we ask when two links will admit orientation preserving homeomorphic C-complexes.  In the case of 2-component links, we find that the pairwise linking number provides a complete obstruction.  In the case of links with 3 or more components and zero pairwise linking number, Milnor's triple linking number provides a complete obstruction.
\end{abstract}

\maketitle

\section{Introduction}

In the 1930's, Seifert \cite{Seifert35}  introduced  the study of a knot $K$ via a compact connected oriented surface now called a \textbf{Seifert surface} bounded by $K$.  While any knot admits many different Seifert surfaces, the study of any Seifert surface for a fixed knot results in interesting invariants.  For example, Seifert surfaces are used to compute The Alexander module \cite{Alexander28}, the Levine-Tristram signature function \cite{L5}, and the Conway polynomial \cite{Kauffman81}.  See also \cite{NS03, Seifert35, Seifert50}. 
 
 In \cite{Cooper82}, Cooper  defines a generalization of a Seifert surface called a \textbf{C-complex} (or {clasp-complex}).  Informally, a C-complex is a collection of embedded surfaces in $S^3$ which might intersect each other in clasps.  An example is depicted in Figure~\ref{fig:Ccomplex}. Similar to a Seifert surface, these objects are not themselves invariants of links,  yet out of them many important invariants of links can be understood.  See for example, \cite{Cimasoni2004, CimFlo, Cooper82}.  A formal definition is given in Section~\ref{sect: C-complexes}.  

\begin{figure}
\centering
\includegraphics[width=14em]{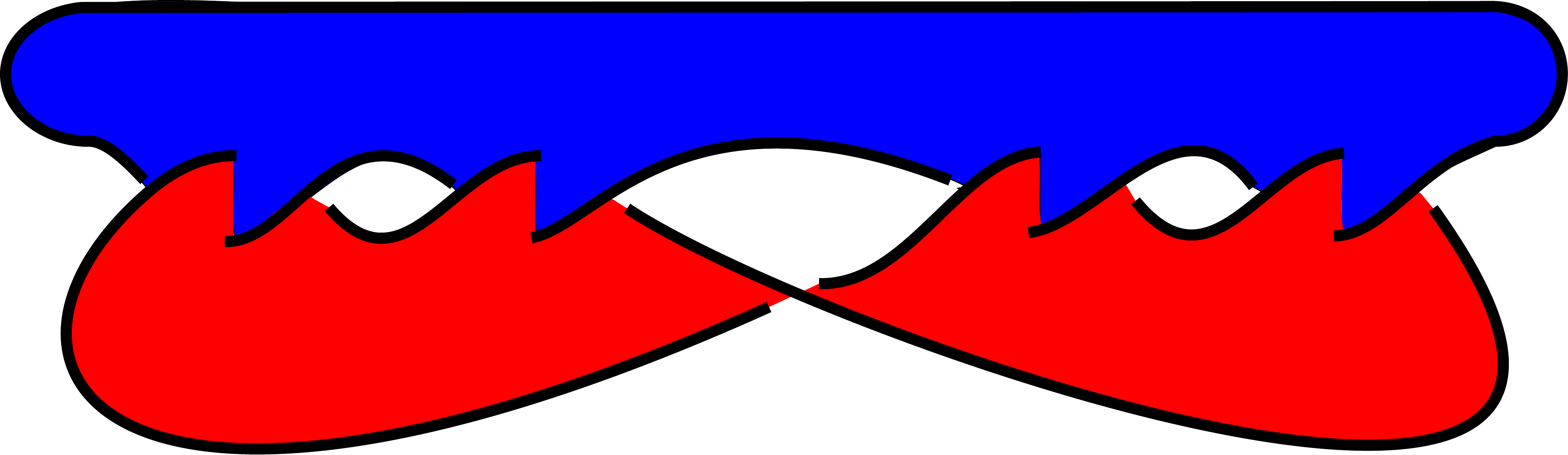}
\caption{An example of a C-complex consisting of two embedded disks which intersect each other in four clasps.}
\label{fig:Ccomplex}
\end{figure}

Notice that for any knots $K_1$ and $K_2$, if $K_1$ bounds a genus $g_1$ Seifert surface $F_1$, $K_2$ bounds a genus $g_2$ Seifert surface $F_2$, and $g_1<g_2$ then by stabilizing $F_1$ as in Figure~\ref{fig:Genus} enough times we create a new Seifert surface $F_1'$ for $K$ such that $g(F_1') = g(F_2)$.  Since the genus and the number of boundary components give a complete set of invariants of compact oriented connected surfaces we conclude that $K_1$ and $K_2$ bound homeomorphic Seifert surfaces.  

\begin{figure}
\begin{picture}(190,30)
\centering
\put(0,0){\includegraphics[height=2em]{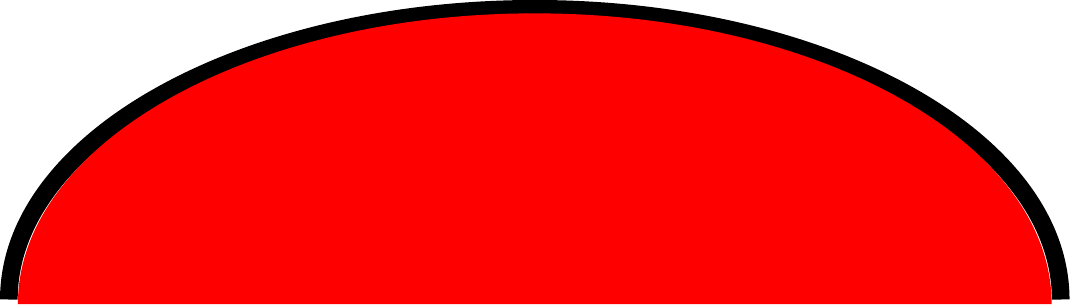}}
\put(90,5){\huge{$\rightsquigarrow$}}
\put(130,0){\includegraphics[height=2em]{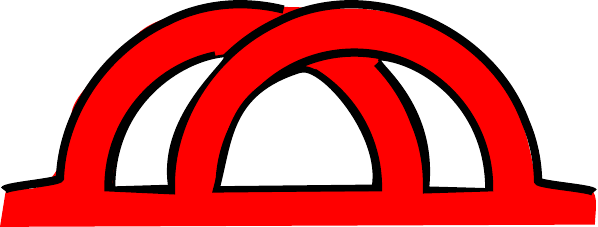}}
\end{picture}
\caption{Stabilizing to increase the genus of a surface.}
\label{fig:Genus}
\end{figure}

The goal of this paper is to ask when the same is true of C-complexes.  Given two $n$-component links $L$ and $J$, when do there exist C-complexes $F$ and $G$ for these links which are related by an orientation preserving homeomorphism?  To what extent is the homeomorphism type of a C-complex an invariant of the link?  We will call two C-complexes $F=F_1\cup \dots \cup F_n$ and $G=G_1\cup \dots \cup G_n$ \textbf{equivalent} if there is a homeomorphism $\Phi:F\to G$ which restricts to orientation preserving homeomorphisms from the components of $F$ to the components of $G$ and preserves the signs of the clasps.  See Definition~\ref{def: equiv} in Section~\ref{sect: C-complexes} for more detail.

%

Since the pairwise linking numbers of a link can be computed by counting clasps in a C-complex with sign, it is clear that the pairwise linking number is an obstruction to two links admitting equivalent C-complexes.  In the case of 2-component links linking number is the only obstruction.

\begin{theorem}\label{thm: 2-components}
Let $L=L_1 \cup L_2$ and $J=J_1 \cup J_2$ be 2-component links.  Then  $\lnk(L_1 , L_2) = \lnk(J_1,J_2)$ if and only if $L$ and $J$ admit equivalent C-complexes.
\end{theorem}

In \cite{M2}, Milnor produced a family of \textbf{higher order linking numbers}.  For an $n$-component link $L$, we will be most interested in the \textbf{triple linking number} $\overline{\mu}_{ijk}(L)\in \Z$ (with $1\le i<j<k\le n$).    In \cite{MellorMelvin2003}, Mellor and Melvin give a means of computing Milnor's triple linking number from a C-complex.  We recall this result in Section \ref{sect:triple linking}.    Their formulation depends only on the equivalence class of a C-complex for the link.  Hence we see that triple linking number gives an obstruction to links admitting equivalent C-complexes.  Indeed, in the case of links with vanishing pairwise linking numbers, the triple linking numbers form a complete obstruction.



\begin{theorem}\label{thm:main}
Let $L$ and $J$ be $n$-component links with vanishing pairwise linking numbers.  Then the following are equivalent
\begin{enumerate}
\item For all $1\le i<j<k\le n$, $\overline{\mu}_{ijk}(L)=\overline{\mu}_{ijk}(J)$
\item $L$ and $J$ admit equivalent C-complexes
\item  There exist unknotted curves $\gamma_1,\dots,\gamma_k$ disjoint from $L$ such that $\lnk(L_i,\gamma_j)=0$ for all $i,j$ and $J$ is obtained from $L$ by some surgery on $\gamma_1,\dots,\gamma_k$.
\end{enumerate}
\end{theorem}

The implication (2) $\implies$ (1) follows immediately from the formulation of $\overline \mu_{ijk}(L)$ in \cite{MellorMelvin2003}.  We recall this formulation in subsection \ref{sect:triple linking}.  The implication (1) $\implies$ (3) relies on a necessary and sufficient condition due to Martin \cite[Theorem 1]{MartinThesis} for two links to be related by a sequence of band pass moves.  See Figure \ref{fig:BPSurgery}.

For links with nonvanishing pairwise linking number, $\overline{\mu}_{ijk}$ is only well defined modulo the greatest common divisor of $\lnk(L_i, L_j)$, $\lnk(L_i, L_k)$ and $\lnk(L_j, L_k)$.  On our way to proving Theorem~\ref{thm:main} above we gain the following result.

\begin{proposition}\label{Prop:mu123}
Let $L$ and $J$ be $n$-component links.  If $L$ and $J$ admit equivalent C-complexes, then for all $1\le i<j<k$, $\overline{\mu}_{ijk}(L) = \overline{\mu}_{ijk}(J)$ (modulo the greatest common divisor of $\lnk(L_i, L_j)$, $\lnk(L_i, L_k)$ and $\lnk(L_j, L_k)$).
\end{proposition}

Note that our results fail to address the question posed in the title of this document in the case of links with more than two components and non-vanishing pairwise linking numbers.  The complete solution will require an answer to the following question.

\begin{question}
Let $L$ and $J$ be $n$-component links.  Suppose that $\lnk(L_i,L_j)=\lnk(J_i,J_j)$ for all $i,j$ but that $\lnk(L_i,L_j)\neq 0$ for some $i,j$. Suppose also that $\overline{\mu}_{ijk}(L)=\overline{\mu}_{i,j,k}(L)$ modulo the greatest common divisor of $\lnk(L_i, L_j)$, $\lnk(L_i, L_k)$ and $\lnk(L_j, L_k)$.  Does it follow that $L$ and $J$ admit equivalent C-complexes?
\end{question}

The genus of a knot $K$ is defined to be the minimum genus of all Seifert surfaces for $K$.  Similarly, for a link $L$ one can define $\beta(L)$ to be the minimum first Betti number of all C-complexes for $L$.  The following question asks how this measure of complexity behaves for equivalent C-complexes.

\begin{question}
Suppose that $L$ and $J$ are $n$-component links which admit equivalent C-complexes.  Let $\beta(L,J)$ the the minimum first betti number of all C-complexes for $L$ which are equivalent to some C-complex for $J$.  Do there exist links $L$ and $J$ for which $\beta(L,J)>\max(\beta(L),\beta(J))$?
\end{question}

\subsection{Organization of paper.}

In Section \ref{sect: C-complexes} we state formally the definition of a C-complex and what it means for two C-complexes to be equivalent.  We then study the relationship between linking numbers and C-complexes.  In subsection \ref{sect:2-components} we prove Theorem~\ref{thm: 2-components}.  In Section~\ref{sect:prelim} we recall the meaning of Milnor's triple linking number and briefly recall surgery.  Finally, in Section~\ref{sect: many components} we prove Theorem~\ref{thm:main}.  

\section{C-complexes and linking numbers}\label{sect: C-complexes}


We begin by recalling the definition of a $C$-complex appearing in \cite{CimFlo}.   

\begin{definition}
An $n$-component C-complex $F=F_1\cup\dots\cup F_n$ is a union of compact oriented connected oriented surfaces in $S^3$ such that 
\begin{enumerate}
\item For all $i$, $\bdry F_i$ is a simple closed curve.
\item For $i\neq j$, $F_i\cap F_j$ is a union of embedded arcs running from a point on $\bdry F_i$ to a point on $\bdry F_j$.  These arcs are called \textbf{clasps}. See Figure \ref{fig:clasps}.
\item For $1\le i<j<k\le n$, $F_i\cap F_j\cap F_k = \emptyset$.
\end{enumerate}
\end{definition}

  Given an oriented link $L = L_1\cup\dots \cup L_n$, we say that $F=F_1\cup \dots \cup F_n$ is a C-complex for $L$ if $\bdry F_i = L_i$ for $i=1,\dots n$.  In \cite[Lemma 1]{Cimasoni2004}, Cimasoni proves that every link admits a C-complex.     
  
\begin{figure}
\begin{picture}(320,100)
\put(70,0){\includegraphics[width=8em]{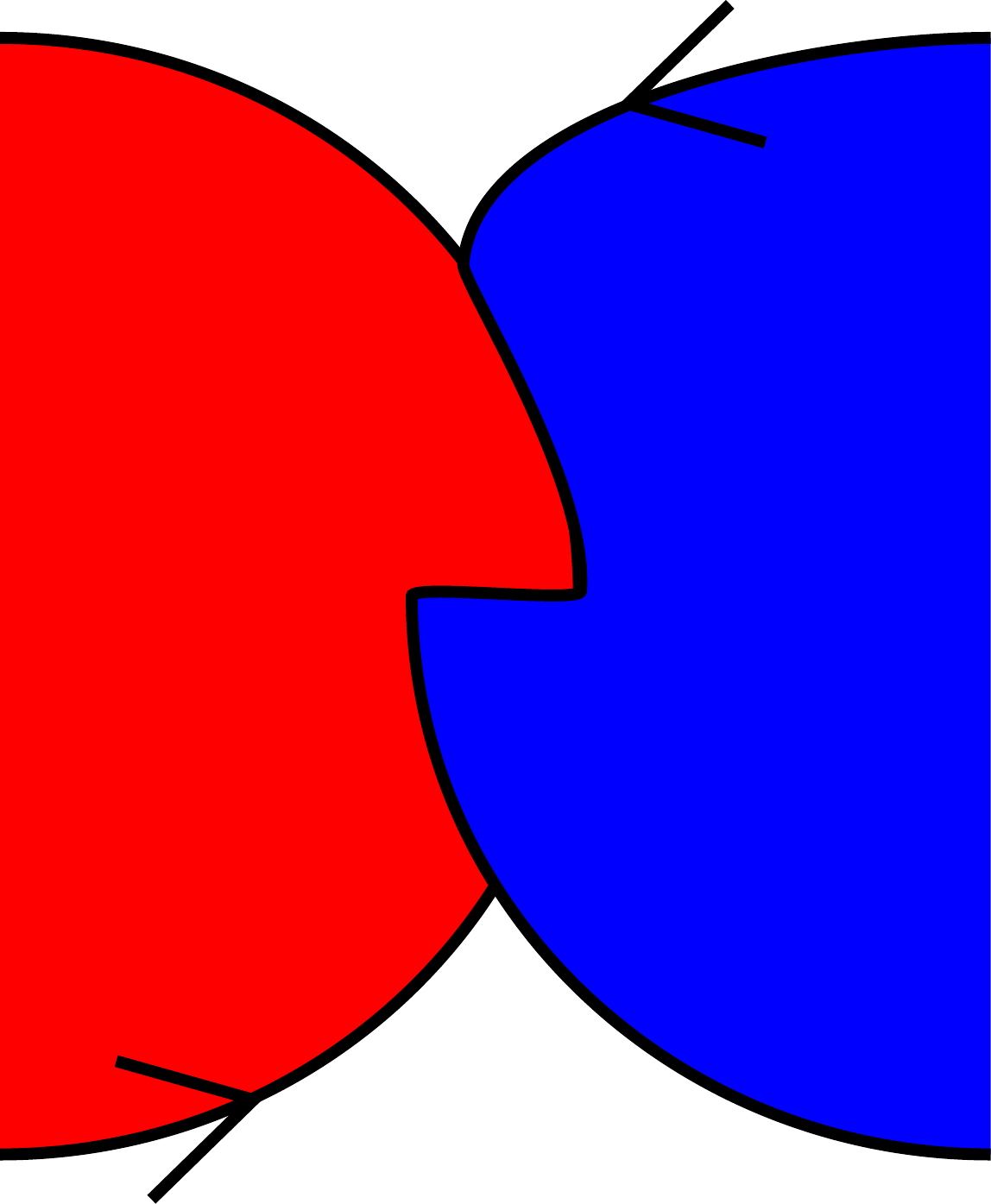}}
\put(200,0){\includegraphics[width=8em]{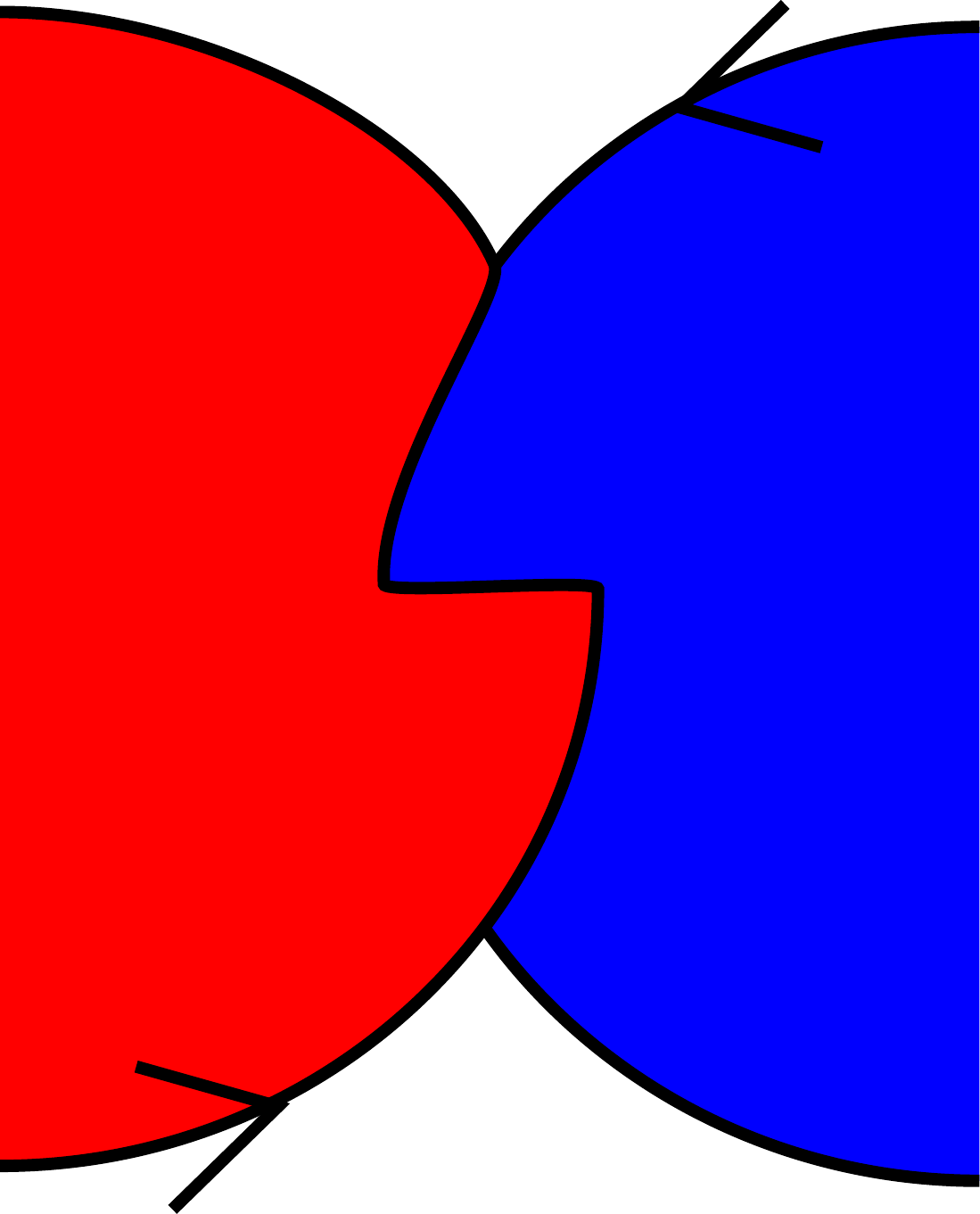}}
\end{picture}
\caption{Left:  A positive clasp. Right: A negative clasp.}
\label{fig:clasps}
\end{figure}

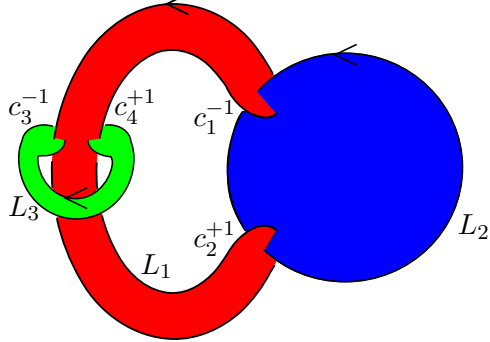
\begin{figure}
\begin{picture}(170,120)
\put(0,0){
\includegraphics[height=12em]{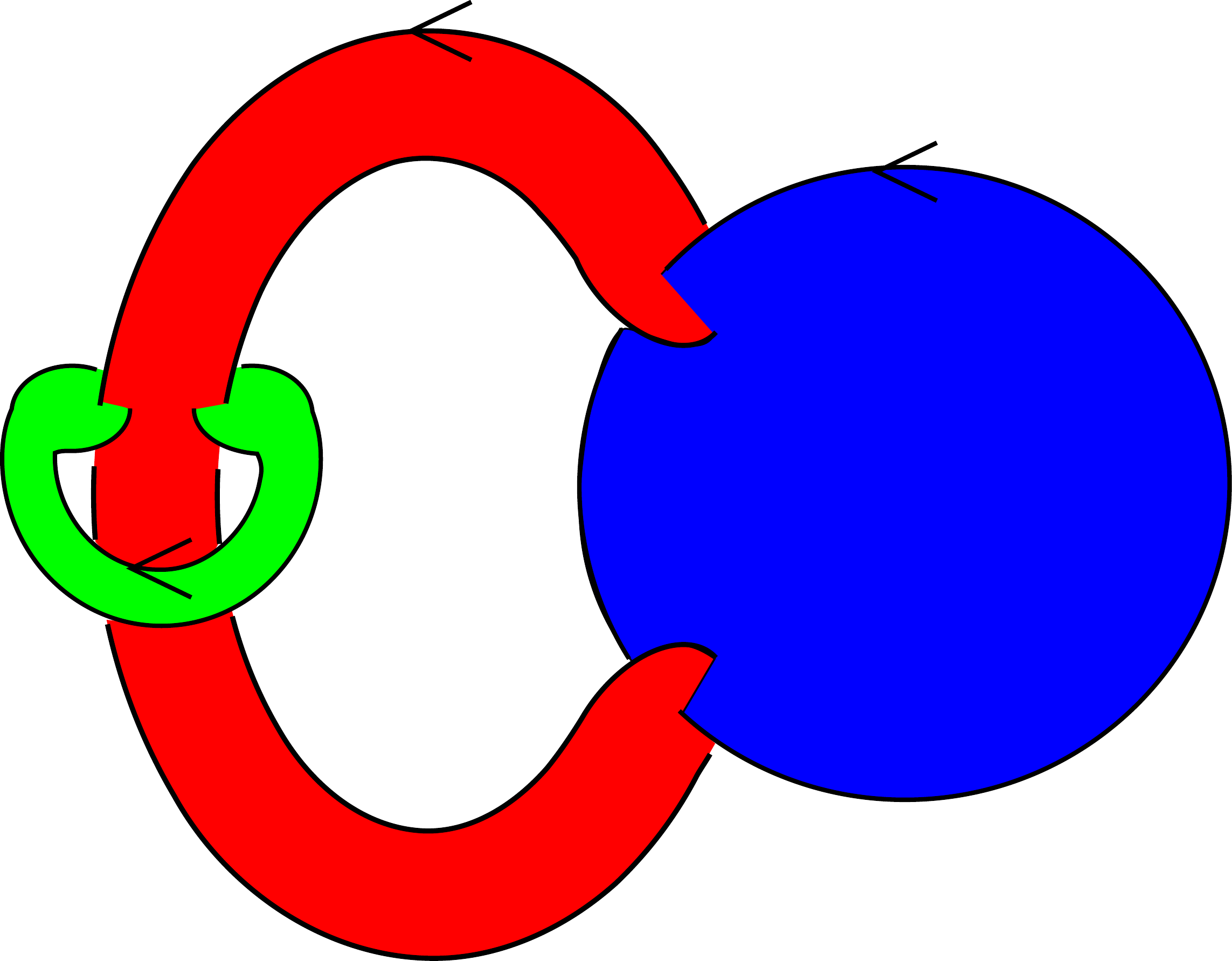}
}
\put(0,47){$L_3$}
\put(50,25){$L_1$}
\put(170,40){$L_2$}

\put(0,85){$c_3^{-1}$}
\put(40,85){$c_4^{+1}$}
\put(70,81){$c_1^{-1}$}
\put(70,35){$c_2^{+1}$}

\end{picture}

\caption{A C-complex for the Boromean Rings with the components and clasps labeled.}
\label{fig:BR}
\end{figure}

\begin{definition}
\label{def: equiv}
Two C-complexes $F = F_1\cup \dots \cup F_n$ and $G = G_1\cup \dots \cup G_n$ in $S^3$ are called \textbf{equivalent} if there exists a homeomorphism $\Phi:F\to G$ such that 
\begin{enumerate}
\item For $k=1,\dots, n$, the restriction $\Phi|_{F_k}$ is an orientation preserving homeomorphism from $F_k$ to $G_k$
\item For every clasp $\c\subseteq F_a\cap F_b$, $\Phi(c) \subseteq G_a\cap G_b$ is a clasp with the same sign as $c$.
\end{enumerate} 
The homeomorphism $\Phi$ is called an \textbf{equivalence} between $F$ and $G$.
\end{definition}

\begin{remark}
   If $N(F_i)$ and $N(G_i)$ are regular neighborhoods of $F_i$ and $G_i$, then the condition that the signs of the clasps agree implies that $\Phi$ extends  to an orientation preserving homeomorphism  from $N(F_1)\cup\dots\cup N(F_n)$ to $N(G_1) \cup\dots\cup N(G_n)$.  
\end{remark}

We proceed by discussing the classification of C-complexes up to this notion of equivalence.  For a C-complex $F = F_1\cup \dots \cup F_n$ bounded by $L_1,\dots, L_n$ let $g(F_k)$ be the genus of $F_k$, the $k$'th component of $F$.  Since the genus of a surface is an invariant of that surface, $g(F_k)$ is an invariant of the equivalence class of $F$.  

Another invariant can be seen by recording the clasps of $F$.  For a C-complex $F$, let $\{c_1,\dots c_k\}$ refer to the set of clasps.  Any clasp $c_i\subseteq F_a\cap F_b$ is assigned a sign $\epsilon_i = \sgn(c_i) = \pm1$ depending on the intersection between $L_a$ and $F_b$ at $c_i$.  See Figure~\ref{fig:clasps}.  In order to encode the sign of the clasps in the notation, we will say $c_i^{\epsilon_i}$ in place of $c_i$.  {For example, consider the C-complex for the Boromean rings of Figure \ref{fig:BR}.  It has clasp set $\{c_1^{-1}, c_2^{+1}, c_3^{-1}, c_4^{+1}\}$.  }

  After picking a basepoint $p_k\in \bdry F_k$ away from the clasps one can build a word $\omega_k(F)$ in the letters $\{c_1^{\epsilon_1},\dots,c_\ell^{\epsilon_\ell}\}$ by following the boundary $L_k$ of $F_k$ and recording $c_i^{\epsilon_i}$ whenever $L_k$ passes through the clasp $c_i^{\epsilon_i}$.  We call these \textbf{claspwords}.  Notice that each clasp $c_i^{\epsilon_i}\in \C^F$ appears in precisely two of these claspwords in which it appears once.  A change of basepoint alters $\omega_k$ by a cyclic permutation.  Notice that since the assignment of labels to the clasps was arbitrary, we can change $\{\omega_k(F)\}_{k=1}^n$ by any permutation of the names of the clasps.

For example, consider again the C-complex for the Boromean rings in Figure~\ref{fig:BR}.  It has 
$$
\omega_1(F)=c_1^{-1}c_3^{-1}c_2^{+1}c_4^{+1},~
\omega_2(F)=c_1^{-1}c_2^{+1},~
\omega_2(F)=c_3^{-1}c_4^{+1},~g(F_1)=g(F_2)=g(F_3)=0.
$$

These invariants give a complete description of the equivalence of C-complexes.

\begin{proposition}\label{prop:equiv C-comp}
Let $F=F_1\cup \dots \cup F_n$ and $G=G_1\cup \dots \cup G_n$ be $n$-component C-complexes.  Then $F$ is equivalent to $G$ if and only if for all $k$,  $g(F_k)=g(G_k)$ and $\omega_k(F)=\omega_k(G)$ (up to a cyclic permutation and relabeling the clasps.)  
\end{proposition}
\begin{proof}

  If $\Phi:F\to G$ is an equivalence, then $\Phi|_{F_k}$ is an orientation preserving homeomorphism.  Since genus is an invariant of surfaces, $g(F_k)=g(G_k)$.   Suppose that $\omega_k(F) = c^{\epsilon_{a_1}}_{a_1}\dots c^{\epsilon_{a_k}}_{a_k}$.   Since $\Phi$ restricts to a homeomorphism from $\bdry F_k$ to $\bdry G_k$ and preserves the signs of the clasps, $\omega_k(F) = \Phi(c_{a_1})^{\epsilon_{a_1}}\dots \Phi(c_{a_k})^{\epsilon_{a_k}}$.  Thus, up to a relabeling of the clasps, $\omega_k(F) = \omega_k(G)$.


Now suppose that for some labeling of the clasps, some choice of basepoints and all $k$,  $\omega_k(F)=\omega_k(G)$ and  $g(F_k) = g(G_k)$.  Let $A_k(F)$ be a closed annular neighborhood of $\bdry F_k$ containing all of the clasps in $F_k$ and $F_k^0\subseteq F_k$ be the closure of the complement of $A_k(F)\subseteq F_k$.  Since  $\omega_k(F)=\omega_k(G)$, there is a homeomorphism $\Phi^A_k: A_k(F)\to A_k(G)$ preserving the clasps.  Since $F_k^0$ and $G_k^0$ are surfaces with the same genus and one boundary component each, the restriction $(\Phi^A_k)|_{\bdry F_k^0}$ extends to a homeomorphism $F^0_k\to G^0_k$.  Thus, we have a homeomorphism $\Phi_k:F_k\to G_k$.  

The map $\Phi$ given by $\Phi(x)= \Phi_k(x)$ for $x\in F_k$ gives an equivalence between $F$ and $G$.
\end{proof}

If $L = L_1\cup\dots \cup L_n$ is an $n$-component link with C-complex $F_1\cup \dots \cup F_n$, then it is clear that $g(F_k)$ and $\omega_k(F)$ are definitely not invariants of $L$.  For instance, the move of Figure \ref{fig:clasppass} merely isotopes the underlying link, yet alters both the claspword and the genus.   In order to build an obstruction theory to a pair of links admitting equivalent C-complexes, we find quantities depending on the claspwords and genera of a C-complex which are invariants of the link.  One such quantity is the linking number $\lnk(L_i, L_j)$.  Recall that $\lnk(L_i, L_j)$ can be computed in terms of bounded surfaces.  See for example, \cite[Chapter 5, Section D]{Rolfsen}.  In particular, if $F_j$ is a surface bounded by $L_j$ then $\lnk(L_i, L_j)$ is given by counting  the number of positive intersections between $L_i$ and $F_j$ and then subtracting the number of negative intersections. This is the same as counting the number of positive clasps in $F_i\cap F_j$ and then subtracting the number of negative clasps.  This clearly depends only on the claspword $\omega_i(F)$.  

\subsection{The proof of Theorem~\ref{thm: 2-components}}\label{sect:2-components}

In this subsection, we prove that if a pair of 2-component links have the same linking number, then these links admit equivalent C-complexes.

\begin{figure}
\begin{picture}(320,100)
\put(70,0){\includegraphics[width=8em]{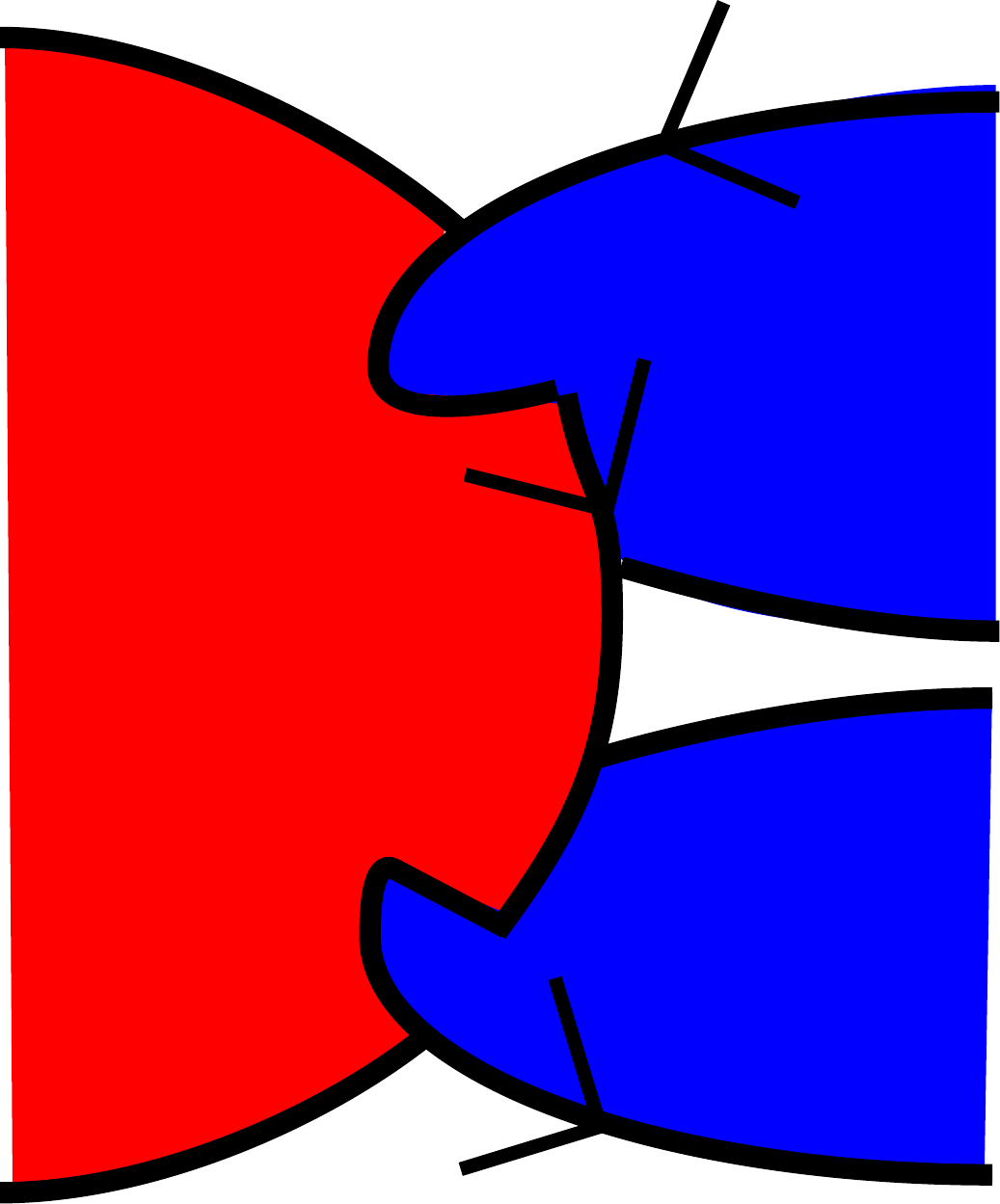}}
\put(170,40){\huge{$\rightsquigarrow$}}
\put(200,0){\includegraphics[width = 10em]{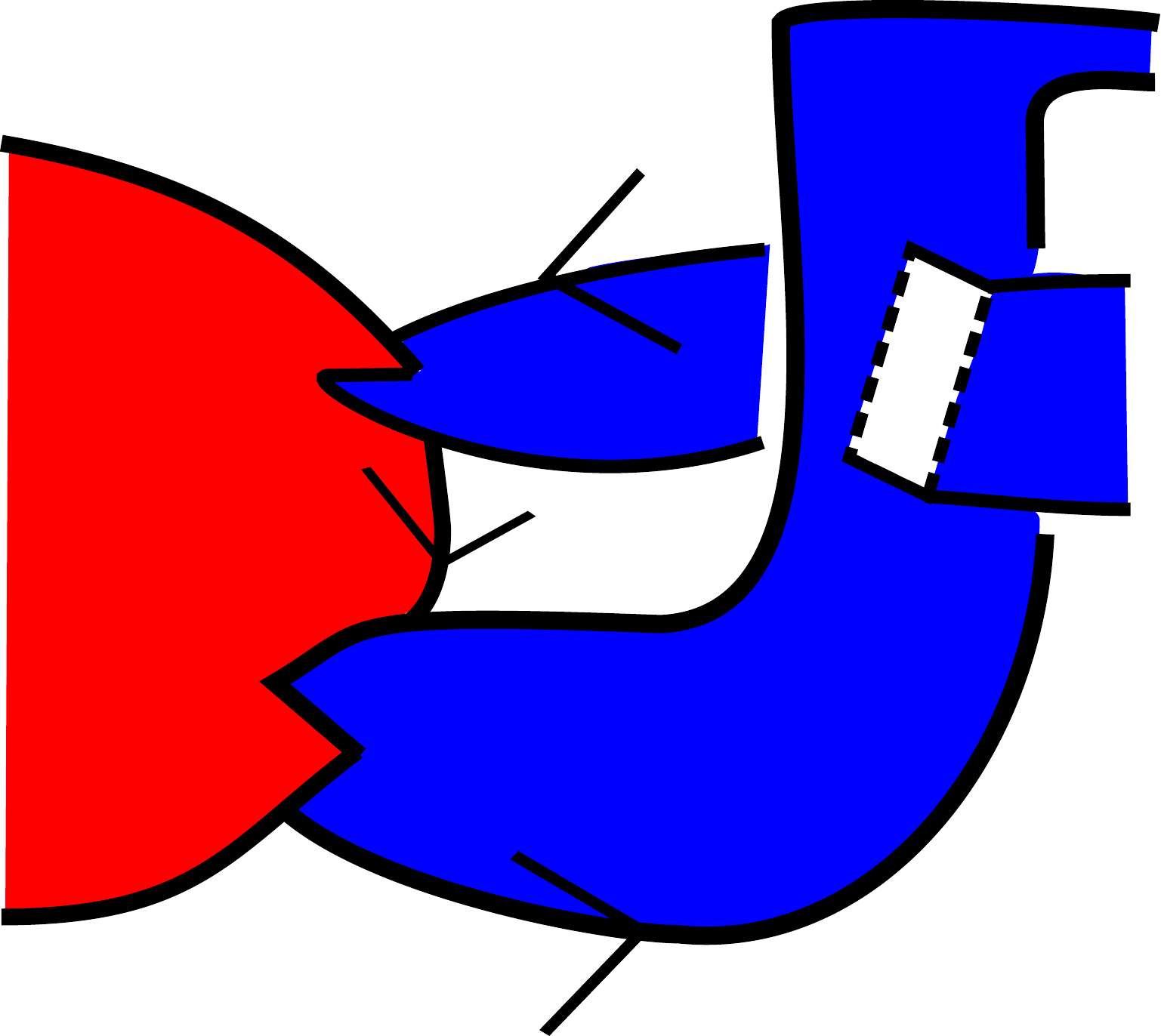}}
\end{picture}
\caption{By isotoping $L=L_1 \cup L_2$, we modify $\omega_1$ by the transposition of two consecutive clasps.  Notice that this increases the genus of $F_2$ by 1.  }\label{fig:clasppass}
\end{figure}

\begin{reptheorem}{thm: 2-components}
Let $L=L_1 \cup L_2$ and $J=J_1 \cup J_2$ be 2-component links.  Then  $\lnk(L_1 , L_2) = \lnk(J_1,J_2)$ if and only if $L$ and $J$ admit equivalent C-complexes.
\end{reptheorem}
\begin{proof}
That the linking number provides an obstruction is trivial, as the linking number can be computed in terms of the number of positive and negative clasps.  It suffices then to suppose that $\lnk(L_1,L_2) = \lnk(J_1,J_2)$ and construct equivalent C-complexes for $L$ and $J$.  Consider any two C-complexes $F = F_1\cup F_2$ and $G=G_1\cup G_2$ for $L$ and $J$ respectively.   We begin by modifying these C-complexes so that their claspwords become very simple.

Suppose that $F$ has $m$ positive and $n$ negative clasps.  Then $F$ has a total of $k=m+n$ clasps.   Since each of these clasps must then be between $F_1$ and $F_2$, every clasp appears once in $\omega_1(F)$ and once in $\omega_2(F)$.  By labeling the clasps according to the order they appear in $L_1$ we arrange that 
$$\omega_{1}(F)=c_{1}^{\epsilon_1}\dots c_k^{\epsilon_k} 
\text{ and  }
\omega_{2}(F) = c_{\sigma(1)}^{\epsilon_{\sigma(1)}}\dots c_{\sigma(k)}^{\epsilon_{\sigma(k)}}$$
for some permutation $\sigma$.  

Consider the move of Figure~\ref{fig:clasppass}. It transposes two adjacent clasps in $\omega_1(F)$ and does not change $\omega_2(F)$. Therefore, $\omega_{1}(F)$ can modified to read $c_{\rho(1)}^{\epsilon_{\rho(1)}} \dots c_{\rho(k)}^{\epsilon_{\rho(k)}}$ for any  permutation, $\rho$.  Pick $\rho$ such that $\epsilon_{\rho(1)} = \dots = \epsilon_{\rho(m)} = +1$ and $\epsilon_{\rho(m+1)} = \dots = \epsilon_{\rho(k)} = -1$.   By relabeling the clasps we now have that 
 $$\omega_{1}(F)=c_{1}^{+1}\dots c_{m}^{+1}c_{m+1}^{-1} \dots c_k^{-1}.$$
 We can now similarly permute $\omega_2(F)$ without altering $\omega_1(F)$ until 
  $$\omega_{2}(F)=c_{1}^{+1}\dots c_{m}^{+1}c_{m+1}^{-1}\dots c_k^{-1}.$$
Similarly we arrange that  $$\omega_{1}(G) = \omega_2(G) = c_{1}^{+1}\dots c_{m'}^{+1}c_{m'+1}^{-1} \dots c_{k'}^{-1}$$
for some $m',k'\in \N$


By modifying $F$ or $G$ as in  Figure~\ref{fig:SL}, we increase the number of positive and negative clasps until $m=m'$.  Since $L$ and $J$ have identical linking numbers, it must then follow that $F$ and $G$ have the same number of negative clasps also.  Thus, we have that $k=k'$.  Notice then that $\omega_1(F)=\omega_1(G)$ and $\omega_2(F)=\omega_2(G)$.  Finally, using the modification in Figure~\ref{fig:Genus} we may assume that $g(F_1)=g(G_1)$ and $g(F_2)=g(G_2)$.  Proposition~\ref{prop:equiv C-comp} now allows us to conclude that $F$ and $G$ are equivalent.

\begin{figure}
\begin{picture}(320,100)
\put(70,0){\includegraphics[height=8em]{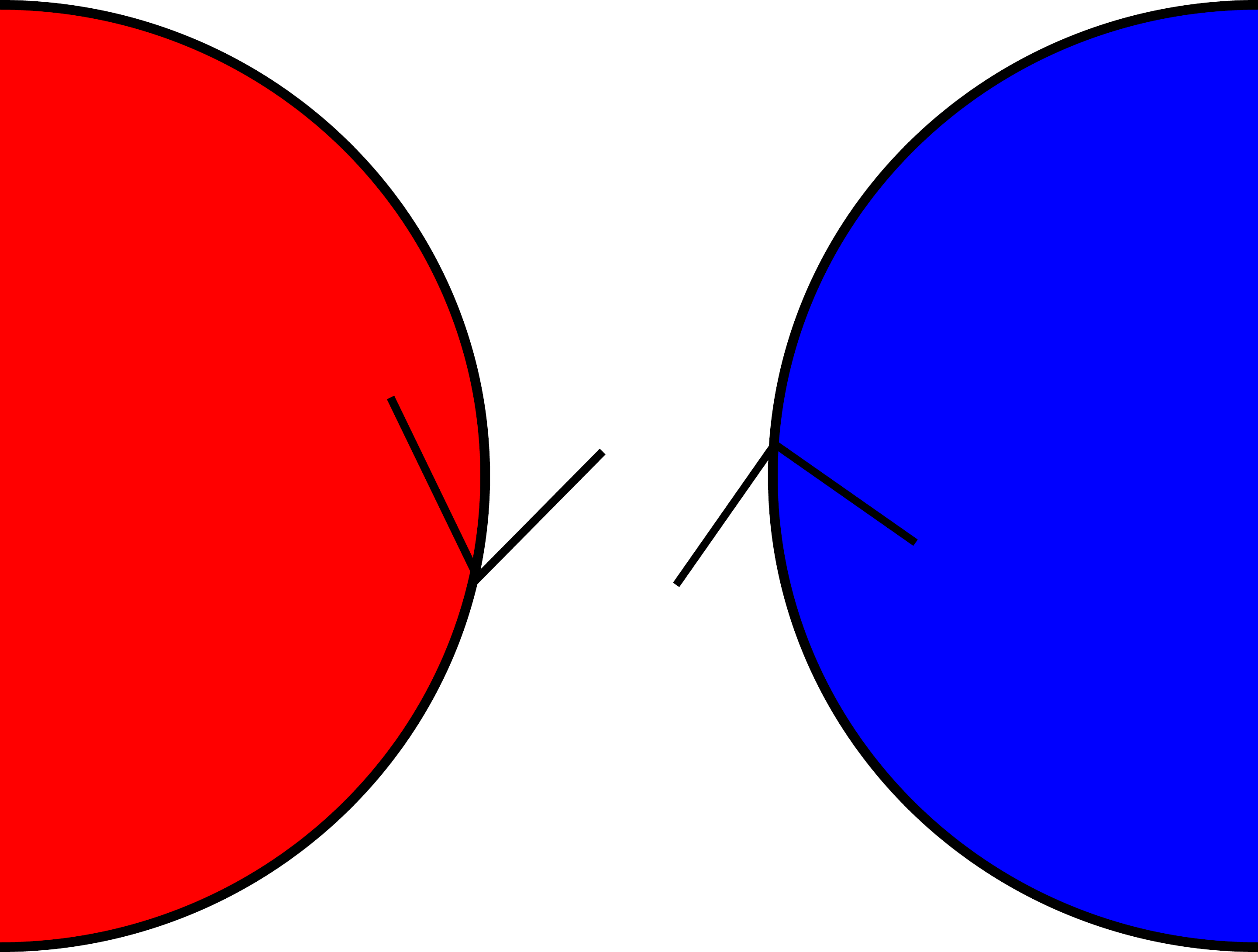}}
\put(190,40){\huge{$\rightsquigarrow$}}
\put(215,0){\includegraphics[height=8em]{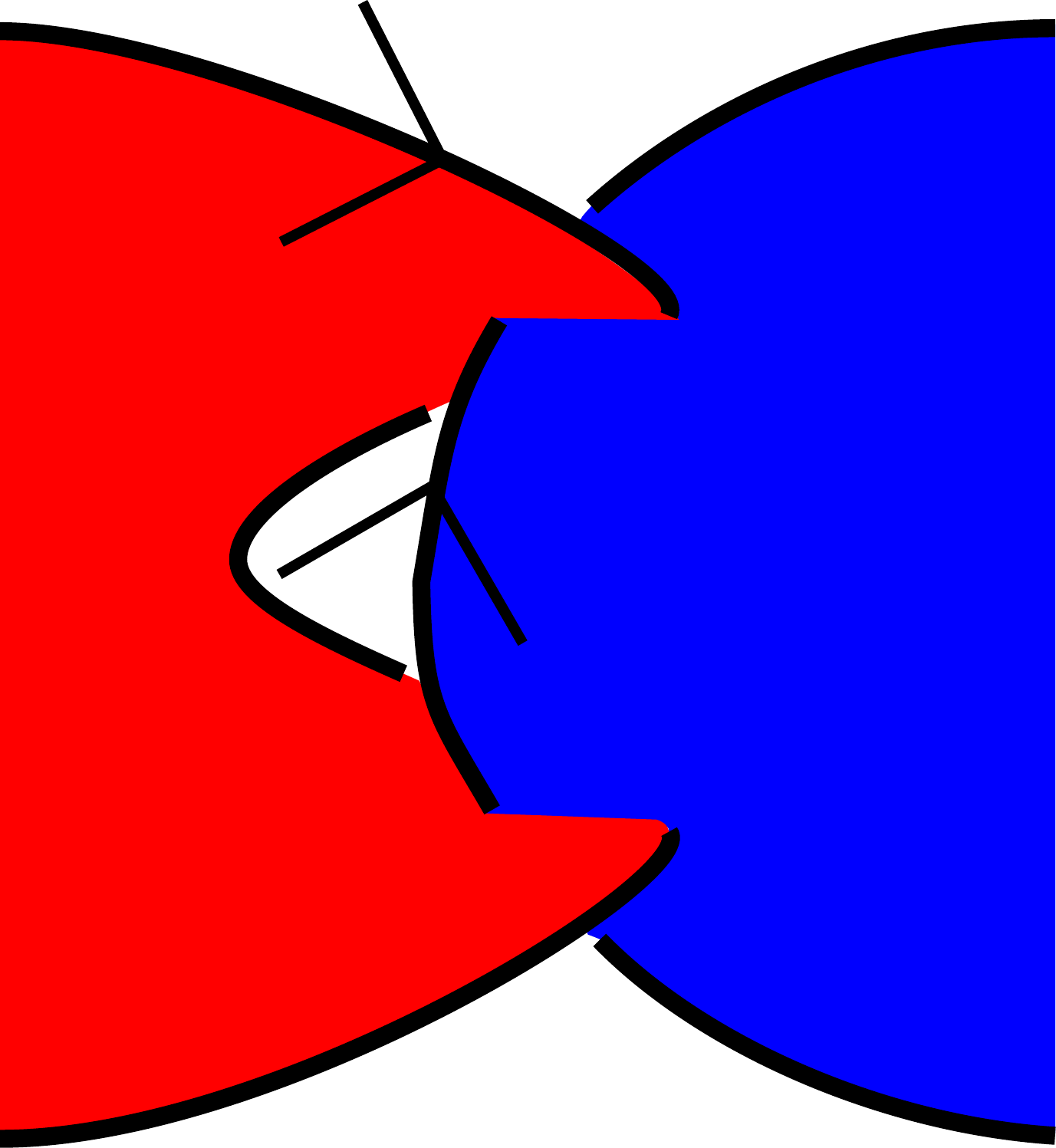}}
\end{picture}

\caption{Addition of a pair of canceling clasps.}
\label{fig:SL}
\end{figure}

\end{proof}

\section{Preliminaries to the proof of Theorem~\ref{thm:main}}\label{sect:prelim}

In this section we recall the concepts needed for our study of C-complexes for links with more than two components.   More precisely, we recall a reformulation of Milnor's triple linking number due to Mellor and Melvin \cite{MellorMelvin2003}.  We go on to informally discuss the notion of surgery.  Since our proof also makes use of the Arf invariant of knots and the Sato-Levine invariant of 2-component links we recall these also in this section.

\subsection{Milnor's triple linking number}\label{sect:triple linking}


  Let  $J=J_1 \cup J_2 \cup J_3$ be a 3 component link admitting a C-complex $F=F_1\cup F_2\cup F_3$ with claspwords $\omega_1(F)$, $\omega_2(F)$, $\omega_3(F)$.  Out of these claspwords we generate new words $u_1(F)$, $u_2(F)$, and $u_3(F)$ in the variables $x_1^{\pm1}, x_2^{\pm1}, x_3^{\pm1}$ as follows.  If $c_\ell^{\epsilon_\ell}$ is a clasp between $F_i$ and $F_j$ then in $\omega_i(F)$, replace $c_\ell^{\epsilon_\ell}$ with $x_j^{\epsilon_\ell}$ and in $\omega_j(F)$ replace $c_\ell^{\epsilon_\ell}$ with $x_i^{\epsilon_\ell}$.  The words $u_1(F)$, $u_2(F)$, and $u_3(F)$ are obtained from $\omega_1(F)$, $\omega_2(F)$, and $\omega_3(F)$ by making this replacement for every clasp.  

  Finally, consider $u_1(F)$, $u_2(F)$, and $u_3(F)$ and take their Magnus expansions.  That is, construct the formal power series in non-commuting variables  $M_1(F), M_2(F), M_3(F)\in \Z[[h_1,h_2,h_3]]$ by sending $x_i\mapsto (1+h_i)$ and $x_i^{-1}\mapsto (1-h_i+h_i^2-h_i^3\dots)$.  For $\{i,j,k\}=\{1,2,3\}$, define $\epsilon_{ijk}(F)\in \Z$ as the coefficient in front of $h_ih_j$ in $M_k(F)$.  

\begin{definition}[See Theorem 1 of \cite{MellorMelvin2003}]\label{defn:mu123}
For the 3-component link $J = J_1 \cup J_2 \cup J_3$ admitting C-complex $F=F_1\cup F_2\cup F_3$, Milnor's triple linking number is given by 
$$
\overline{\mu}_{123}(J)=\epsilon_{123}(F)+\epsilon_{312}(F)+\epsilon_{231}(F).
$$
It is well defined as an invariant of $L$ modulo the greatest common divisor of $\lnk(J_1,J_2)$, $\lnk(J_1,J_3)$, and $\lnk(J_2,J_3)$.  For an $n$-component link $L = L_1 \cup L_2 \dots \cup L_n$ and any three distinct numbers $i,j,k\in \{1,\dots,n\}$ consider the 3-component sub-link $L_i \cup L_j \cup L_k$.  Define
$$
\overline{\mu}_{ijk}(L)=\overline{\mu}_{123}(L_i \cup L_j \cup L_k).
$$
\end{definition}

The triple linking number satisfies that 
$
\overline{\mu}_{ijk}(L)=
\overline{\mu}_{jki}(L)=
\overline{\mu}_{kij}(L)=
-\overline{\mu}_{ikj}(L)=
-\overline{\mu}_{kji}(L)=
-\overline{\mu}_{jik}(L)
$,
so that for our purposes it will suffice to restrict to $1\le i<j<k\le n$.  

Notice that Definition~\ref{defn:mu123} depends only on the claspwords, $\omega_1(F)$, $\omega_2(F)$ and $\omega_3(F)$.  In turn, these depend only on the equivalence class of the C-complex, $F$.  As a consequence we see the the triple linking number provides an obstruction to 3-component links admitting equivalent C-complexes.

\begin{repproposition}{Prop:mu123}
Let $L$ and $J$ be $n$-component links.  If $L$ and $J$ admit equivalent C-complexes, then for all $1\le i<j<k$,  $\overline{\mu}_{ijk}(L) = \overline{\mu}_{ijk}(J)$ (modulo the greatest common divisor of $\lnk(L_i, L_j)$, $\lnk(L_i, L_k)$ and $\lnk(L_j, L_k)$).
\end{repproposition}

\begin{example}
For the sake of illustration, we perform this computation for the C-complex $F$ for $B$, the Boromean rings depicted in Figure~\ref{fig:BR}.  
It has clasp-words
$$
\omega_1(F)=c_1^{-1}c_3^{-1}c_2^{+1}c_4^{+1}, ~
\omega_2(F)=c_1^{-1}c_2^{+1},~
\omega_3(F)=c_3^{-1}c_4^{+1}.
$$
Performing the replacement described in Definition~\ref{defn:mu123} we get
$$
u_1(F)=x_2^{-1}x_3^{-1}x_2^{+1}x_3^{+1}, ~
u_2(F)=x_1^{-1}x_1^{+1}=1,~
u_3(F)=x_1^{-1}x_1^{+1}=1
$$ 
and Magnus expansions
$$
\begin{array}{l}
M_1(F)=(1-h_2+\dots)(1-h_3+\dots)(1+h_2)(1+h_3), \\
M_2(F)=1,~
M_3(F)=1.
\end{array}
$$ 
Expanding,
$$
M_1(F)=1-h_2^2-h_3^2-h_3h_2+h_2h_3+\dots, ~
M_2(F)=1, ~
M_3(F)=1.
$$ 
The ellipses denote terms of degree at least $3$.  Hence, 
$$
\epsilon_{123}=0,~\epsilon_{312}=0,~\epsilon_{231}=1
$$ 
and $\overline\mu_{123}(B)=1$.  
\end{example}

Since the three component unlink has $\overline{\mu}_{123}(U)=0$ we have the following corollary

\begin{corollary}
There does not exist a C-complex for the Boromean Rings which is equivalent to a C-complex for the three component unlink.  
\end{corollary}


\subsection{Surgery}  Given a link $\gamma = \gamma_1,\dots,\gamma_k$ and a sequence of rational numbers in reduced form $\frac{p_1}{q_1},\dots, \frac{p_k}{q_k}$ one can form a new 3-dimensional manifold by cutting out tubular neighborhoods of the curves $\gamma_1,\dots, \gamma_k$ and gluing in their place solid tori $V_1,\dots, V_k$ such that for $i=1,\dots, k$ the meridian of $V_i$ is identified to the curve consisting of  $p_i$ meridians of $\gamma_i$ and $q_i$ longitudes.  For a more complete description of surgery the reader is directed to \cite[Chapter 9]{Rolfsen}.

Given a link $L$ disjoint from the surgery curves $\gamma$ of the preceding paragraph then one sees a new link by taking the image of $L$ in the space resulting from the surgery.  If the resulting link is isotopic to $J$, then we say that $J$ is \textbf{obtained} from $L$ by surgery along $\gamma$.  

\begin{figure}
\begin{picture}(320,100)
\put(240,80){-1}
\put(240,5){-1}
\put(70,0){\includegraphics[height=8em]{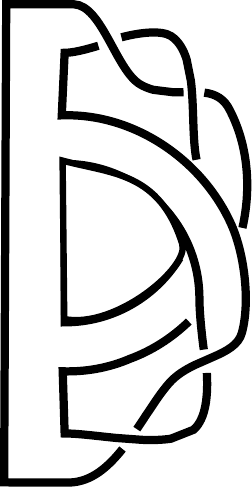}}
\put(200,0){\includegraphics[height=8em]{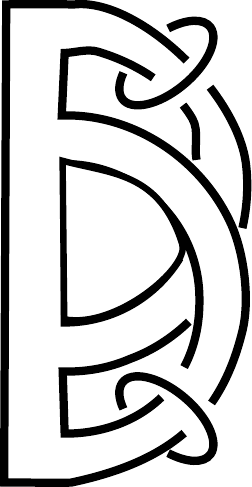}}
\end{picture}
\caption{Left: The trefoil knot.  Right: Performing -1 surgery on these curves in the complement of an unknot produces the trefoil.}
\label{fig:Trefoil}
\end{figure}
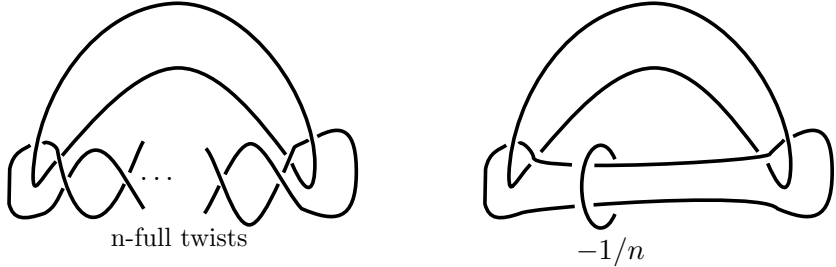
\begin{figure}
\begin{picture}(350,100)
\put(105,25){$\dots$}
\put(94,0){\small{n-full twists}}
\put(0,5){\includegraphics[height=8em]{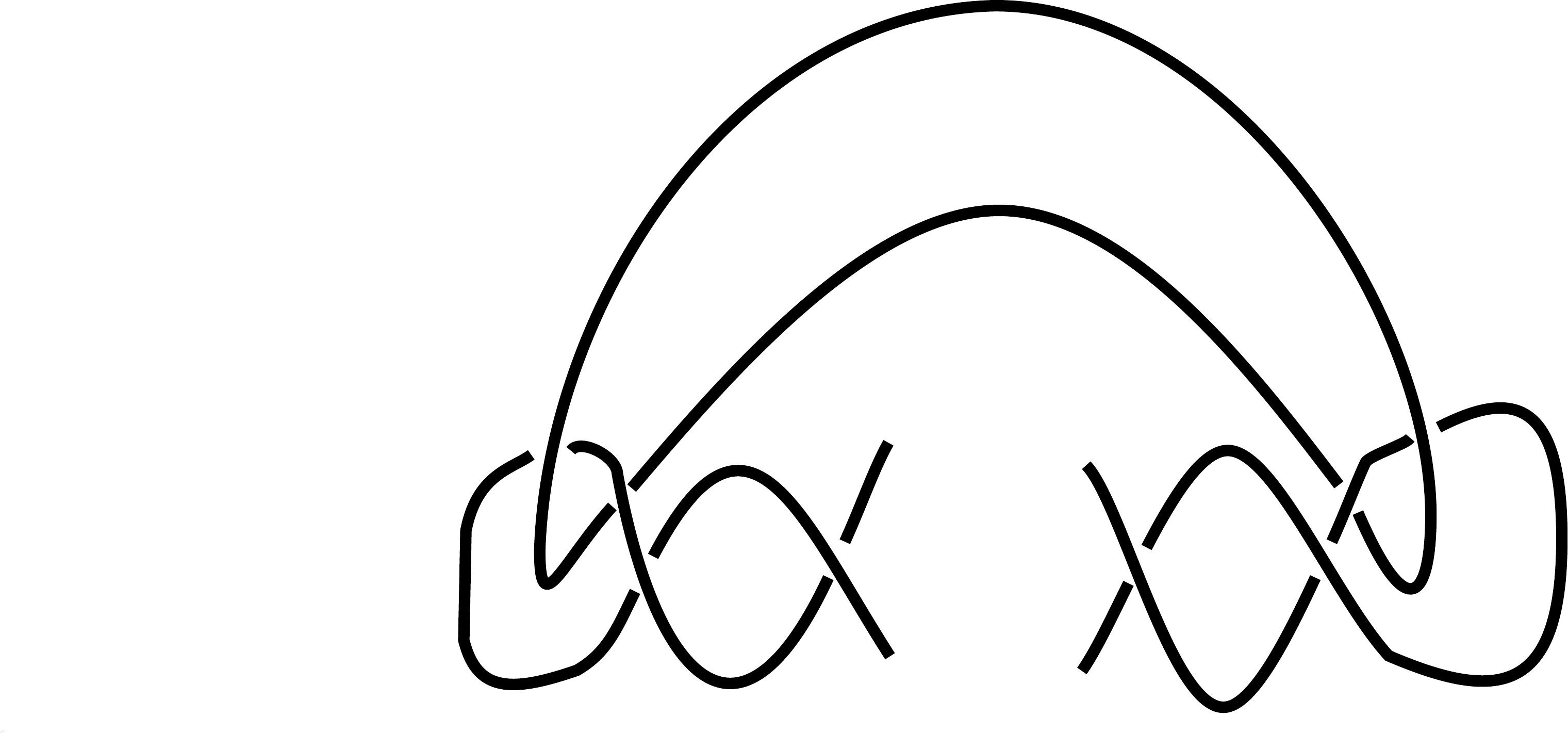}}
\put(180,5){\includegraphics[height=8em]{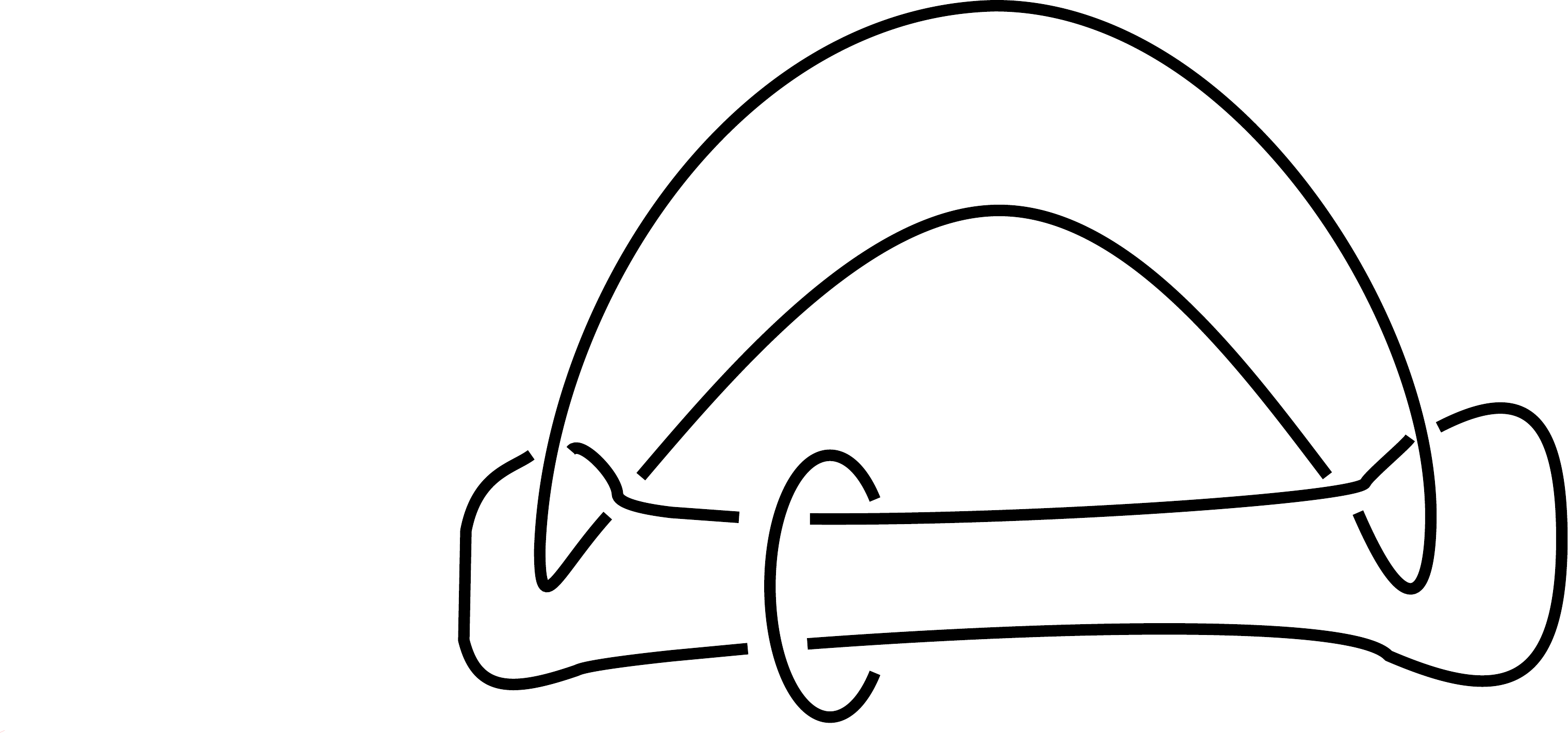}}
\put(270,-5){$-1/n$}
\end{picture}
\caption{Left:  The $n$-twisted Whitehead link.  Right: $\frac{-1}{n}$ surgery along this curve sends the unlink to the $n$-twisted Whitehead link.  }
\label{fig:Whitehead}
\end{figure}

Many interesting surgeries on $S^3$ produce $S^3$.  There are two of particular interest to us in this paper.  First, if $\gamma$ is  unknotted then $\frac{-1}{n}$ surgery along $\gamma$  produces $S^3$ and puts $n$ positive full twists in the strands of $L$ passing through a disk bounded by $\gamma$ \cite[Chapter 9]{Rolfsen}.  For example, the trefoil of Figure \ref{fig:Trefoil} is obtained from the unknot by $-1=\frac{-1}{1}$ surgery along two curves in the complement.  Similarly, the $n$-twisted Whitehead link of Figure~\ref{fig:Whitehead} is obtained from the unlink by a $\frac{-1}{n}$ surgery.   The other move of interest to us is the band pass move of Figure~\ref{fig:BPSurgery}, which can be obtained by $0=\frac{0}{1}$ surgery along two Hopf linked curves.  While this result is well known the proof in its full detail can be found in the proof of \cite[Lemma 1]{MartinThesis}.

\begin{figure}
\begin{picture}(320,100)
\put(120,76){0}
\put(115,10){0}
\put(70,0){\includegraphics[width=8em]{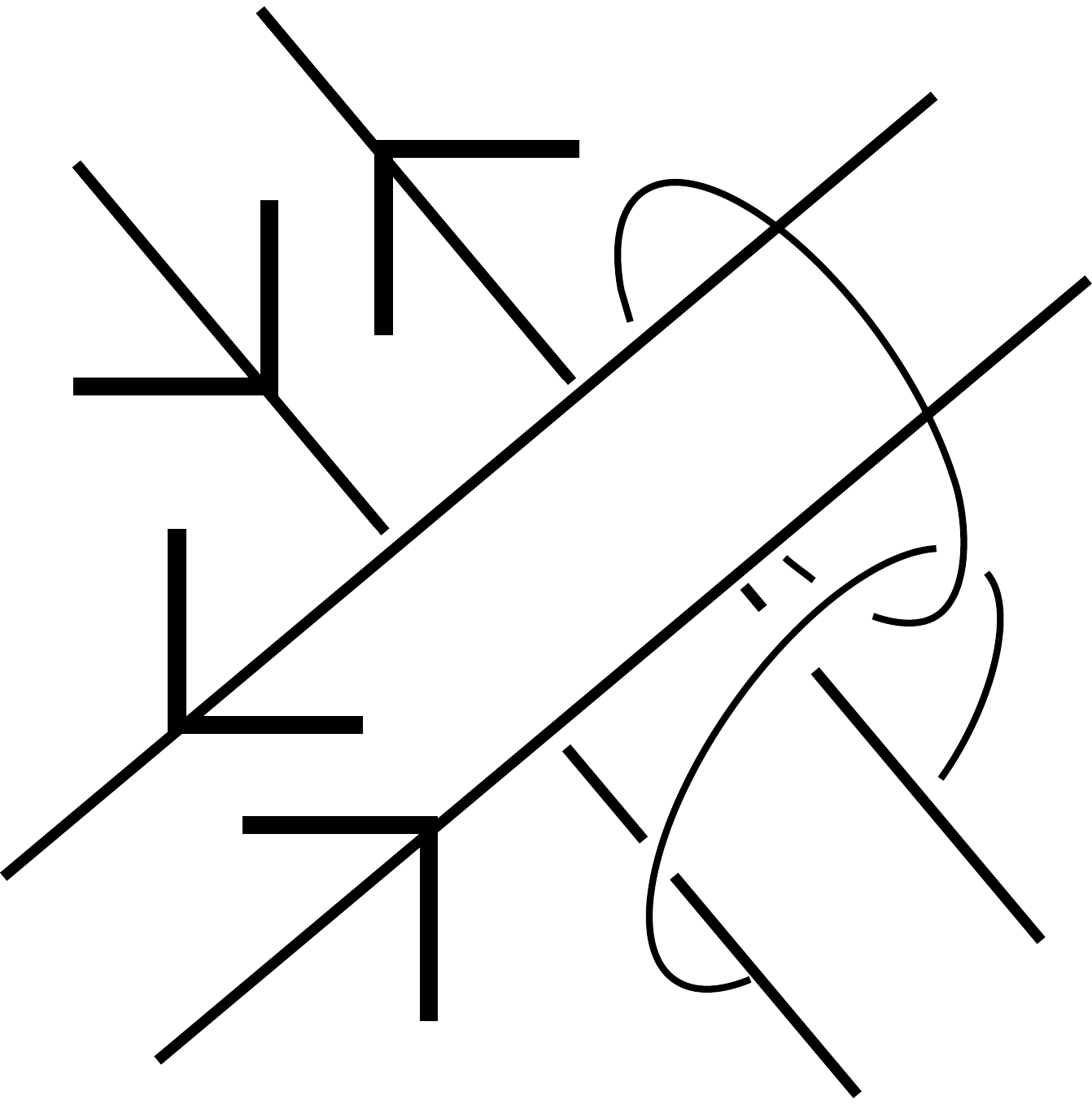}}
\put(200,0){\includegraphics[width=8em]{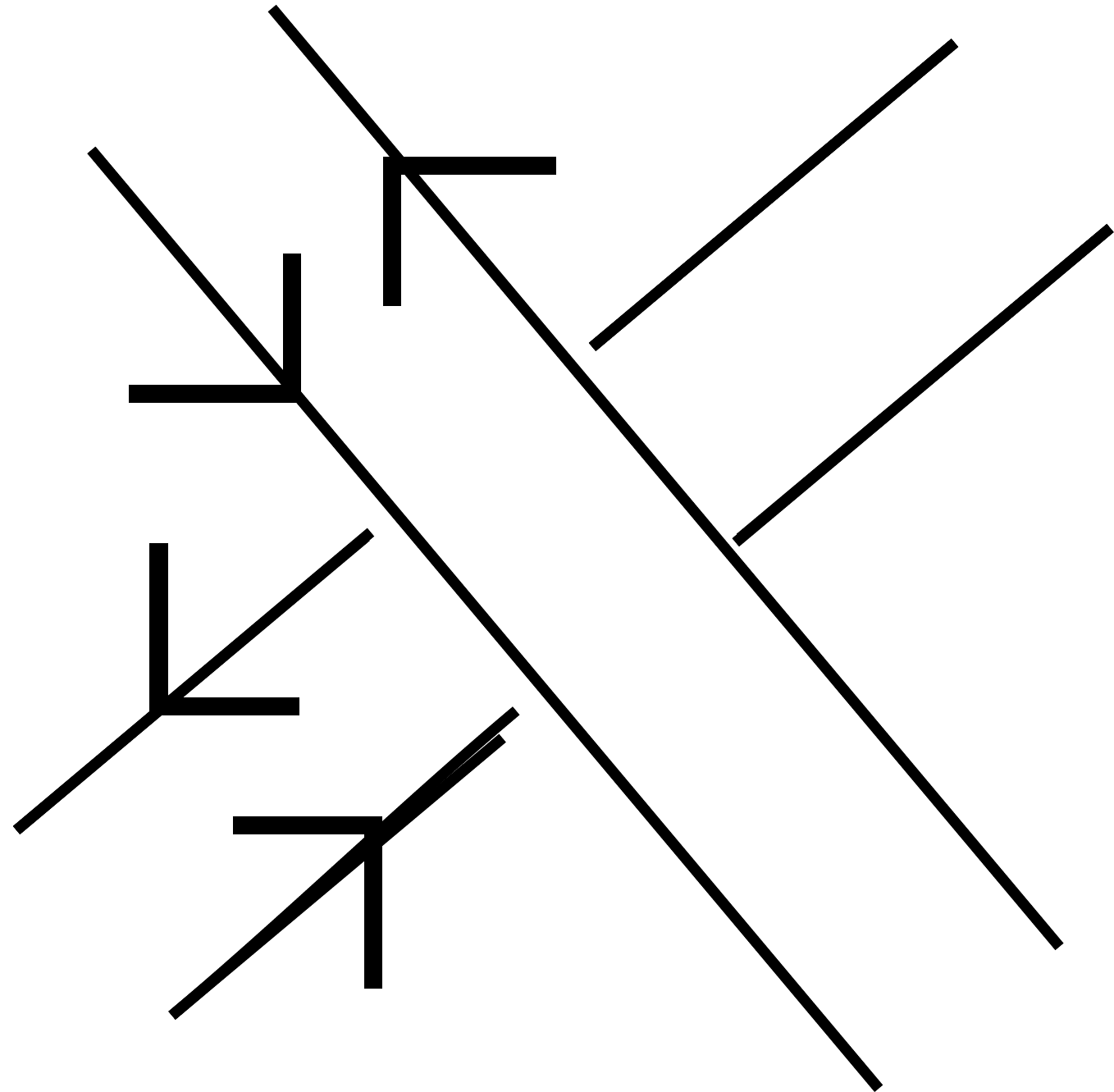}}
\end{picture}
\caption{Left:  A Hopf link in the complement of a link.  Right:  Performing $0$-surgery along the components of this Hopf link accomplishes a band pass.  }
\label{fig:BPSurgery}
\end{figure}

\subsection{The Arf invariant and the Sato-Levine invariant}

The Arf invariant is an invariant of knots which takes values in $\Z/2$.  For a complete discussion of the Arf invariant, see \cite[Chapter X]{OnKnots}.  To see a rigorous discussion on the Sato-Levine invariant for a 2-component link in terms of Milnor's invariant the reader is directed to \cite{C4}.  Our work will require the following elementary properties of these invariants.     

\begin{proposition}\label{Arf}
 For any knot $J$, let the knot $J'$ be obtained by band summing with the Trefoil.  Then the Arf-invariant satisfies that $\Arf(J') = \Arf(J)+1$.
\end{proposition}

\begin{proposition}\label{Mu}
Let $\WH_n$ be the $n$-twisted Whitehead link of of Figure \ref{fig:Whitehead} and let $L=L_1\cup L_2$ be a 2-component link with vanishing pairwise linking number.  Let $L'$ be the 2-component link obtained by band summing with $\WH_n$.  Then the Sato-Levine invariant satisfies that $\overline{\mu}_{1122}(L') = \overline{\mu}_{1122}(L)+n$. 
\end{proposition}

For an $n$-component link $L=L_1\cup \dots\cup L_n$, $\overline{\mu}_{iijj}(L) = \overline{\mu}_{1122}(L_i, L_j)$ is the Sato-Levine invariant of the 2-component sublink $L_i, L_j$.  

According to \cite[Theorem 1]{MartinThesis}, the Arf-invaraint, the Sato-Levine invariant and the triple linking number are a complete set of obstructions to a pair of links being related by a sequence of band pass moves, as in Figure~\ref{fig:BPSurgery}.  

\begin{figure}
\begin{picture}(350,100)
\put(0,5){\includegraphics[height=8em]{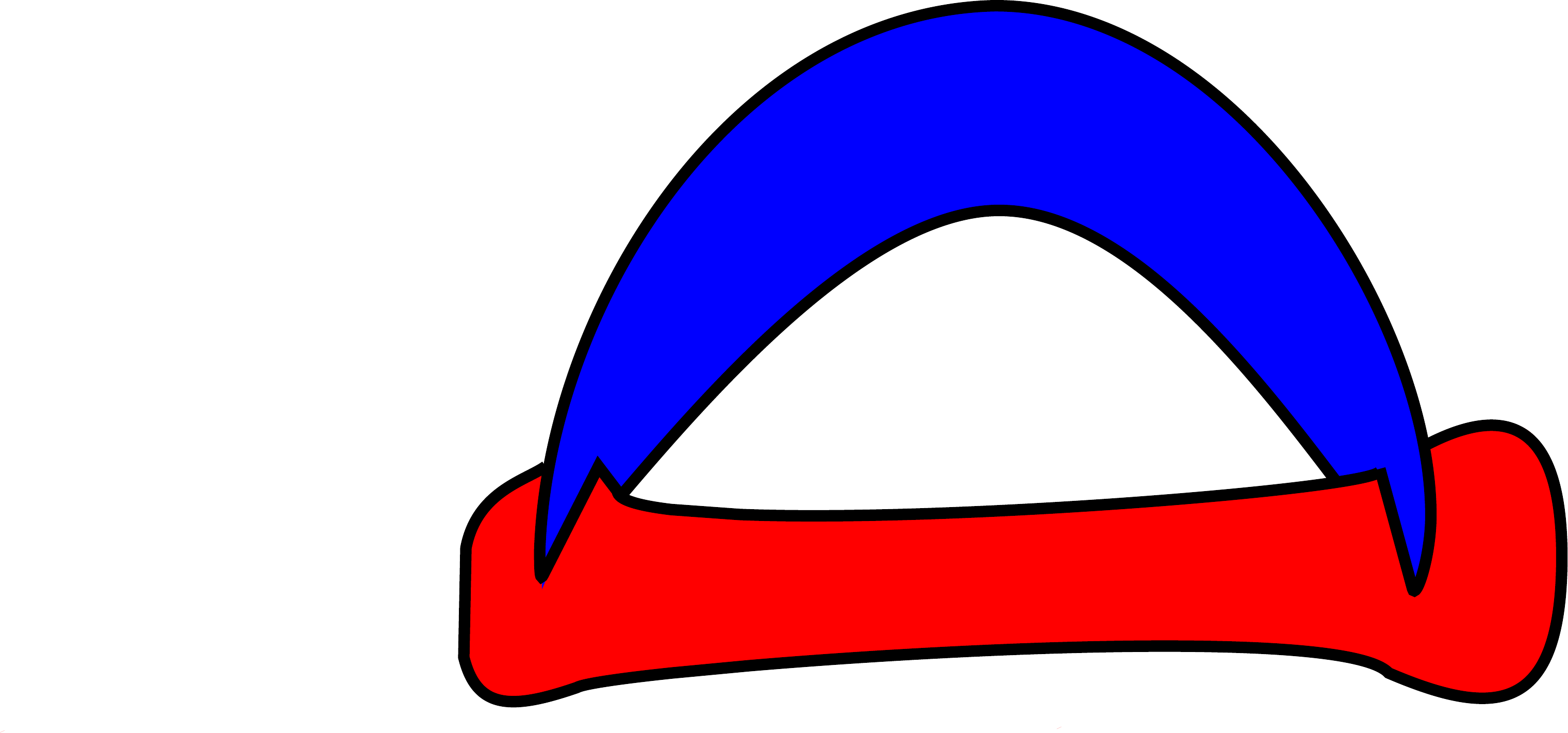}}
\put(290,25){$\dots$}
\put(270,0){$n$ full twists}
\put(180,5){\includegraphics[height=8em]{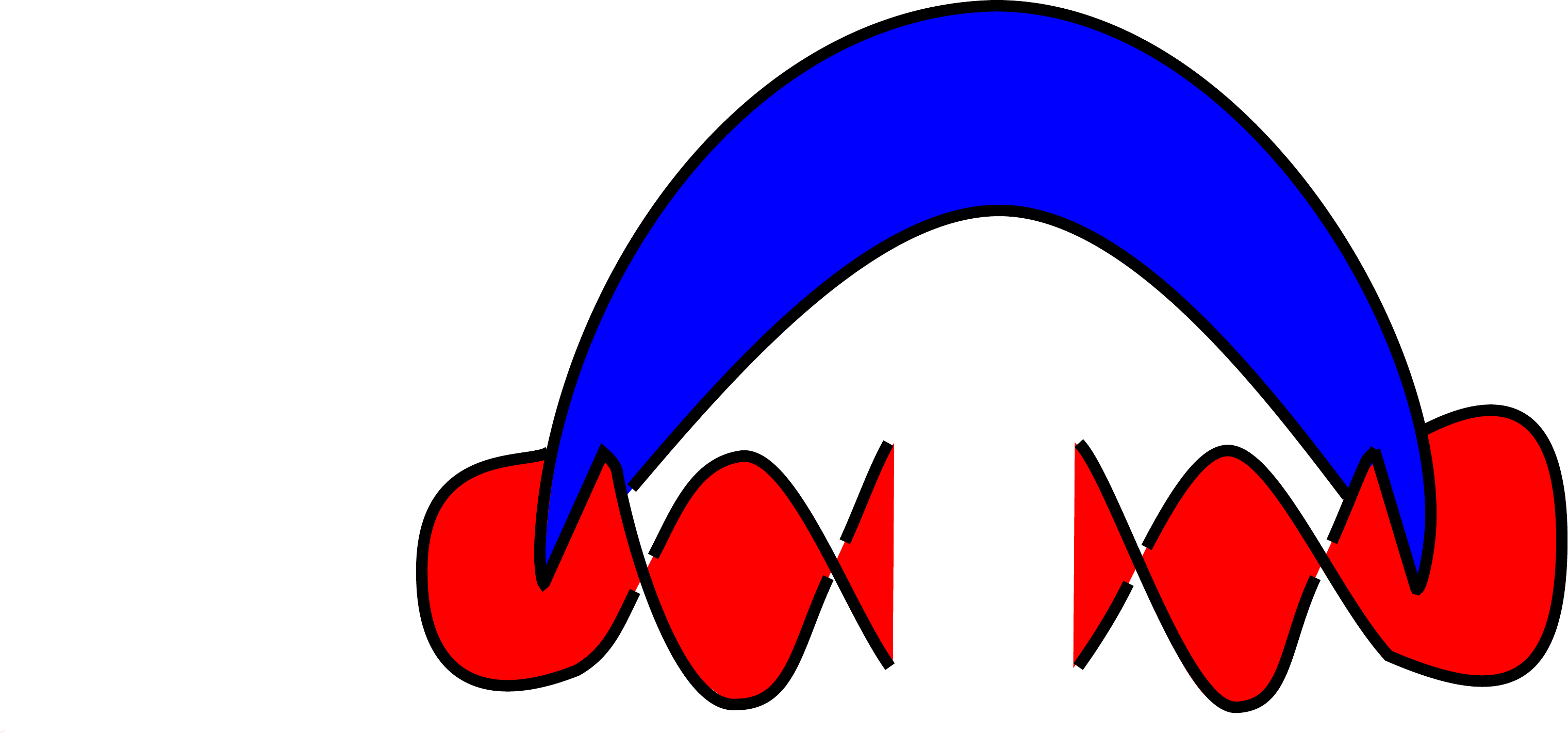}}
\end{picture}
\caption{Left:  A C-complex for the 2-component unlink.  Right: An equivalent C-complex for the $n$-twisted Whitehead link.  }
\label{fig:WhiteheadCplx}
\end{figure}

\section{The proof of Theorem \ref{thm:main}. }\label{sect: many components}

We close this paper with the proof of Theorem \ref{thm:main}.

\begin{reptheorem}{thm:main}
Let $L$ and $J$ be $n$-component links with vanishing pairwise linking numbers.  Then the following are equivalent
\begin{enumerate}
\item For all $1\le i<j<k\le n$, $\overline{\mu}_{ijk}(L)=\overline{\mu}_{ijk}(J)$.
\item $L$ and $J$ admit equivalent C-complexes.
\item There exist curves $\gamma_1,\dots,\gamma_k$ disjoint from $L$ such that $\lnk(L_i,\gamma_j)=0$ for all $i,j$ and such that $J$ is obtained from $L$ after performing some surgery on $\gamma_1,\dots,\gamma_k$.
\end{enumerate}
\end{reptheorem}

\begin{proof}
The claim that  (2) implies (1) is the content of Proposition~\ref{Prop:mu123}.  We begin by showing that (3) implies (2).  Let $L=L_1,\dots, L_n$ and $J=J_1,\dots, J_n$ be $n$-component links.  Suppose that there exist curves $\gamma=\gamma_1,\dots,\gamma_k$   in the complement of $L$, for which $\lnk(L_i,\gamma_j)=0$  and a sequence of rational numbers $p_1/q_1,\dots p_k/q_k$ such that $J$ is obtained by modifying $L$ by $p_j/q_j$ surgery along $\gamma_j$ for all $1\le j \le k$.

Let $F$ be a C-complex for $L$.  Consider any component $F_i$ of $F$ and $\gamma_j$ of $\gamma$.  Since $\lnk(L_i, \gamma_j)=0$ by assumption, either $\gamma_j$ is disjoint from $F_i$ or there exists some pair of intersection points $p,q\in F_i\cap \gamma_j$ with opposite sign such that the arc $\alpha\subseteq \gamma_j$ running from $p$ to $q$ is disjoint from $F_i$. We proceed to modify $F_i$ using a tube following $\alpha$ from $p$ to $q$.   See Figure~\ref{fig: stabilize to disjoint}.  This modification may introduce some simple closed curves in the intersection $F_i\cap F_\ell$ for some $\ell\neq i$.  The resulting collection of surfaces is no longer a C-complex.  By pushing $F_i$ and $F_\ell$ along some arcs as in Figure \ref{fig: fix double loop}, we replace each intersection circle with a pair of clasps.  In doing so we merely isotope $L_i$ and $L_\ell$.  

\begin{figure}[b]
\begin{picture}(190,105)
\put(0,5){\includegraphics[height=8em]{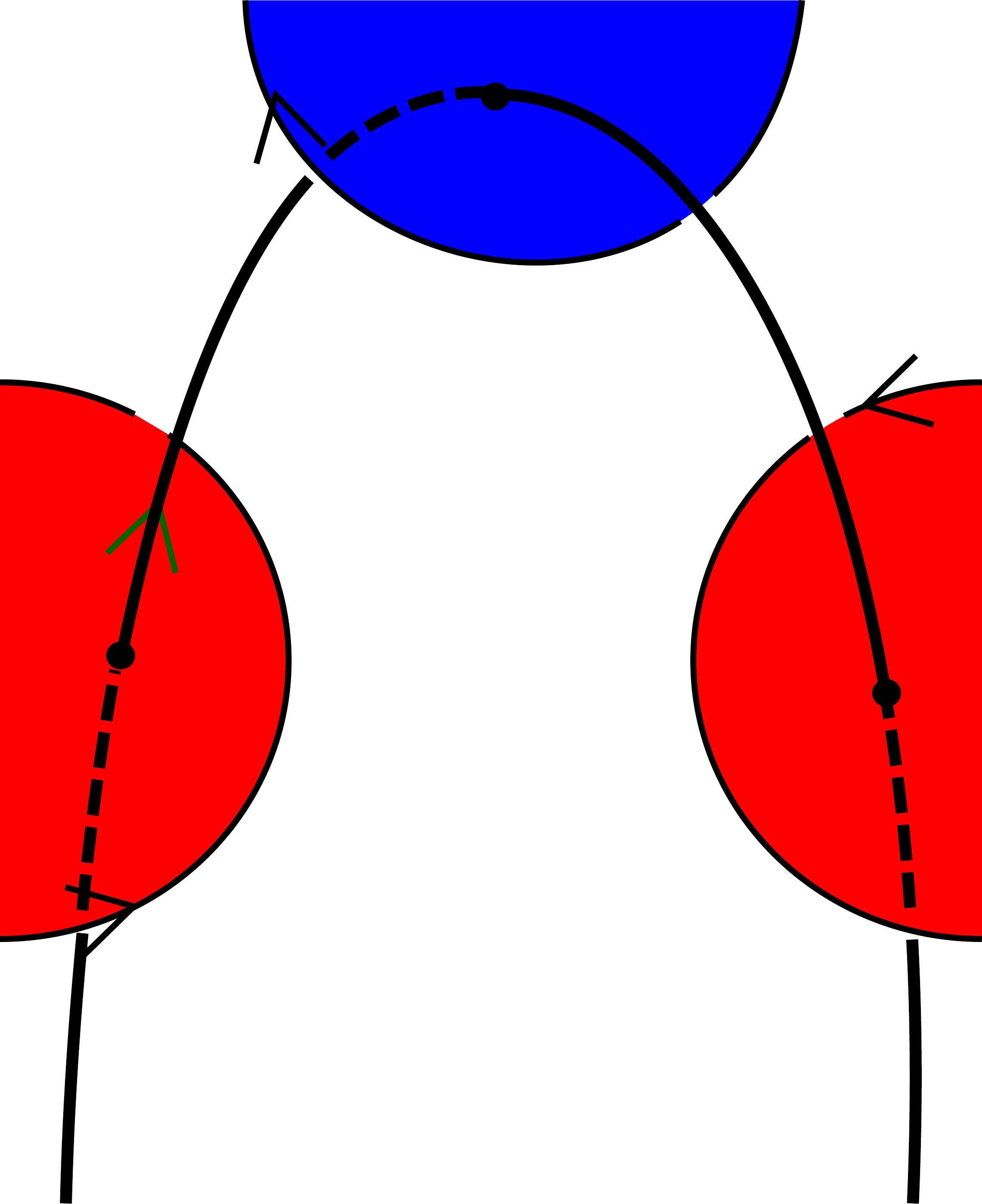}}
\put(6,8){$\gamma_j$}
\put(-13,40){$F_i$}
\put(35,98){$F_\ell$}
\put(30,0){(a)}
\put(120,5){\includegraphics[height=8em]{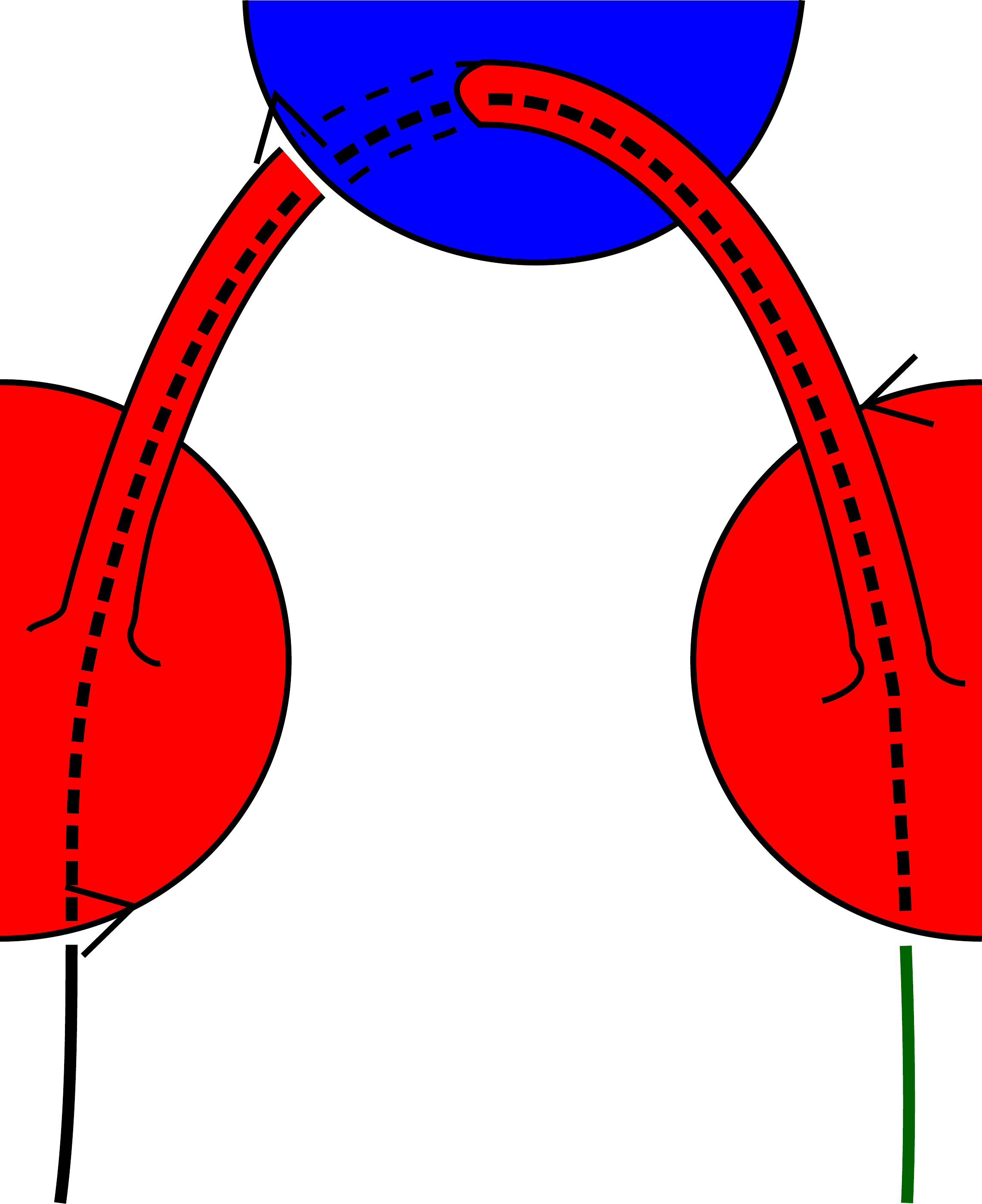}}
\put(126,8){$\gamma_j$}
\put(107,40){$F_i$}
\put(155,98){$F_\ell$}
\put(150,0){(b)}
\end{picture}
\caption{(a) $\gamma_j$ intersects of $F_i$ in two points with opposite sign.  (b) Stabilizing $F_i$ removes these intersections, but adds a simple closed curve in $F_i\cap F_\ell$.}
\label{fig: stabilize to disjoint}
\end{figure}
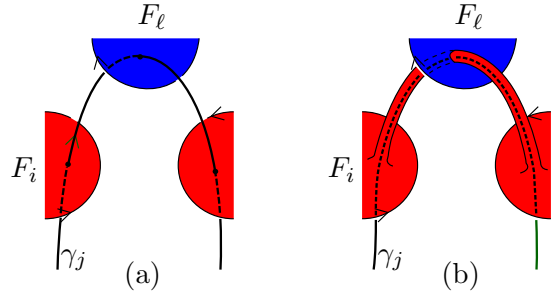

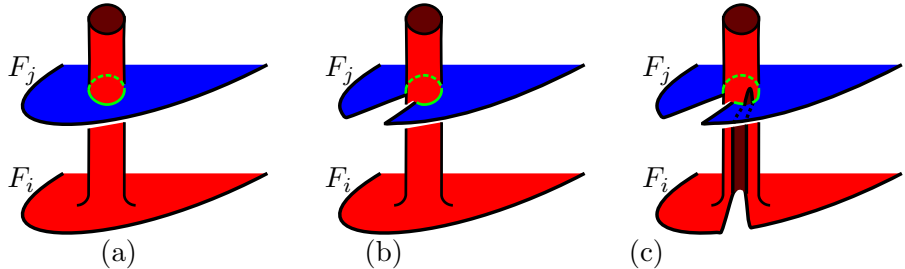
\begin{figure}
\begin{picture}(330,105)
\put(0,10){\includegraphics[height=8em]{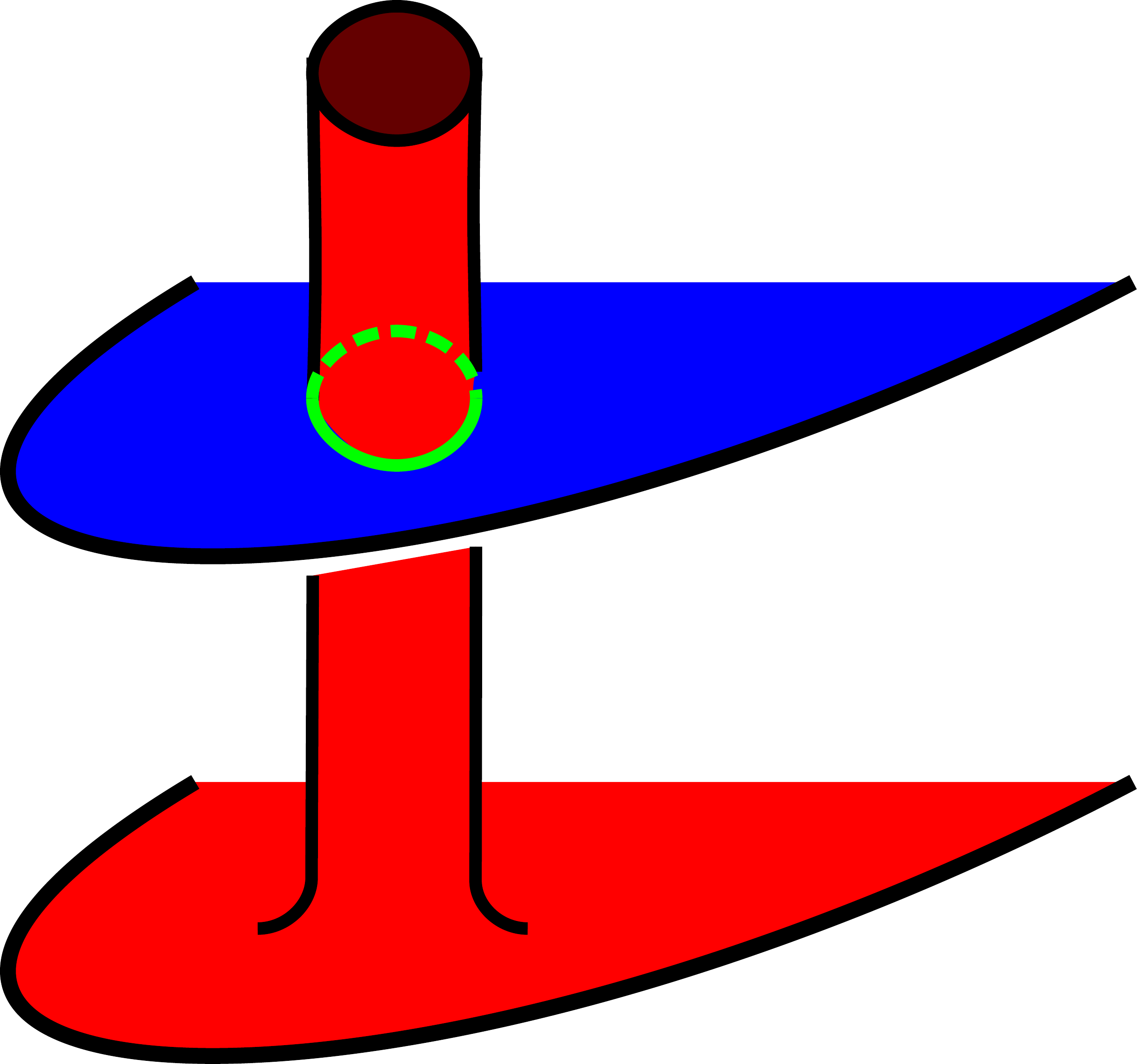}}
\put(-5,28){$F_i$}
\put(-5,70){$F_j$}

\put(30,0){(a)}
\put(120,10){\includegraphics[height=8em]{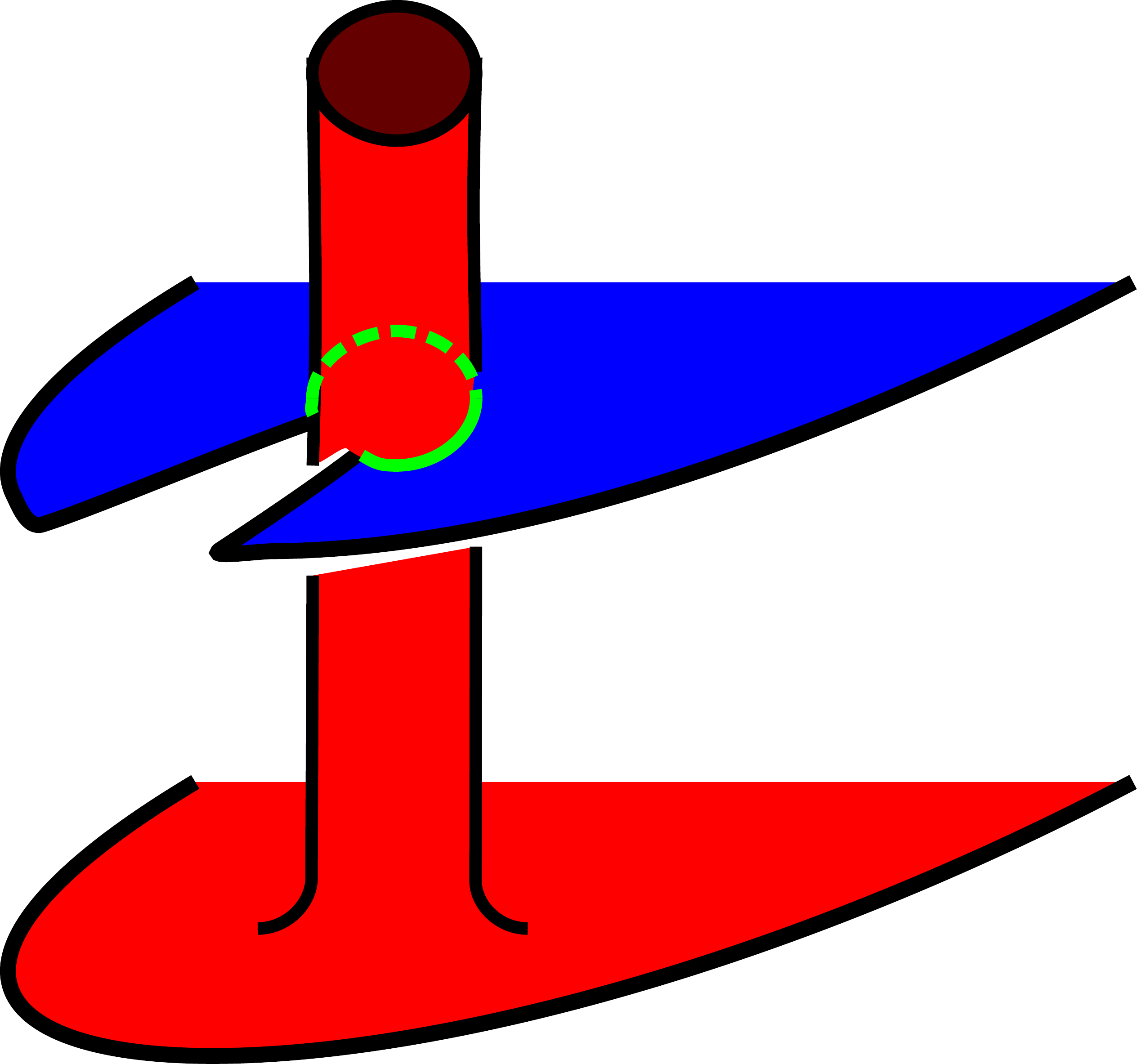}}
\put(115,28){$F_i$}
\put(115,70){$F_j$}

\put(130,0){(b)}
\put(240,10){\includegraphics[height=8em]{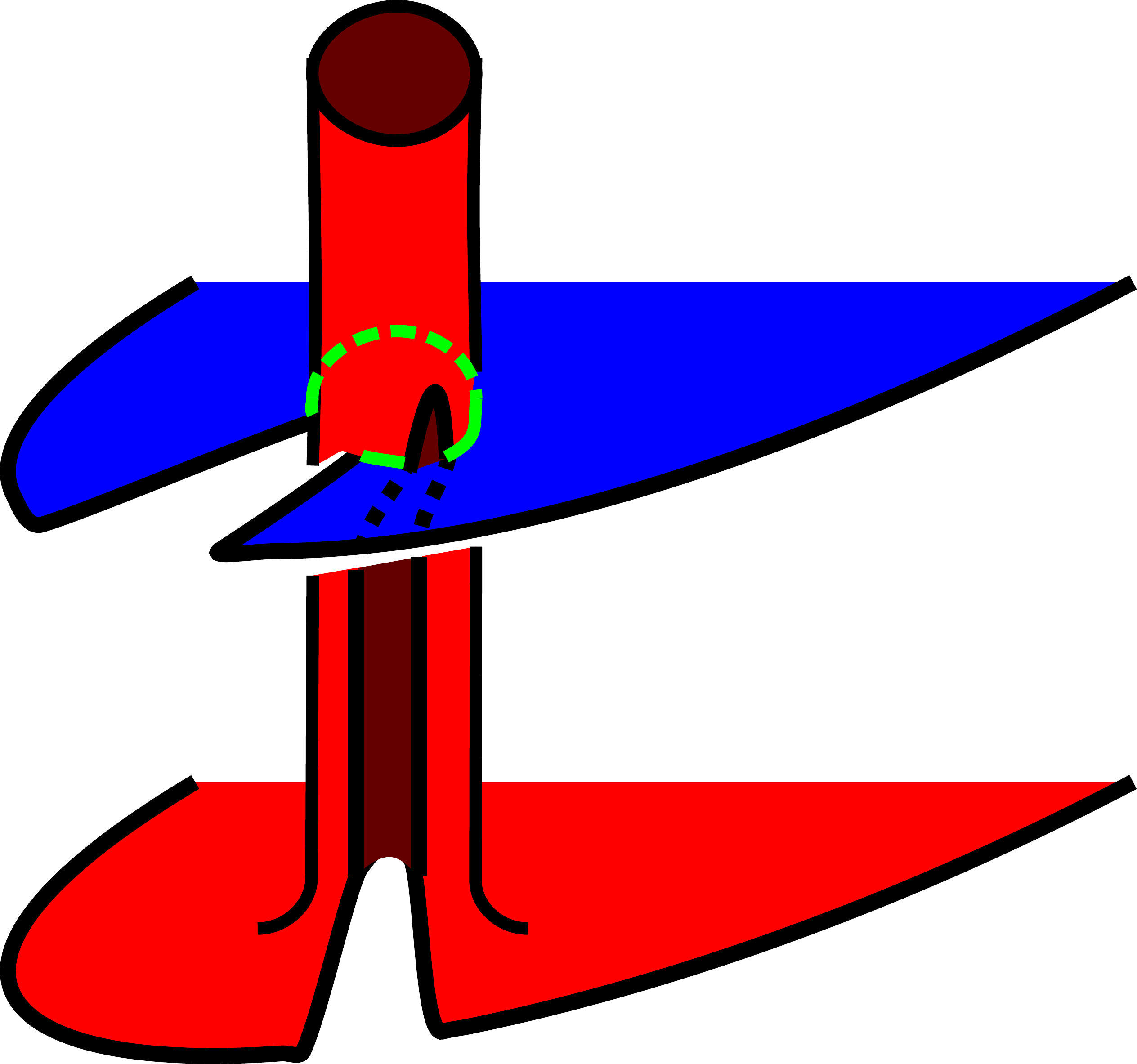}}
\put(235,28){$F_i$}
\put(235,70){$F_j$}

\put(230,0){(c)}
\end{picture}
\caption{(a) A double loop.  The darker red here indicates the opposite side of $F_j$  (b) A ``finger move'' replaces this loop with an arc with both endpoints on $L_j$. (c)  Another ``finger move'' reduces the arc to two clasps.}
\label{fig: fix double loop}
\end{figure}
 
By assumption, the link $J$ is obtained by starting with $L$, removing neighborhoods of $\gamma_1,\dots, \gamma_k$ from $S^3$, and gluing back in solid tori $V_1, \dots, V_k$ such so that the curve given by $p_i$ longitudes of $\gamma_i$ and $q_i$ meridians bounds the meridional disk for $V_i$.  Since $F$ is disjoint from $\gamma$, this cut and paste process preserves the equivalence class of $F$.  We have produced a C-complex for $J$ equivalent to a C-complex for $L$, as required by the theorem.

It remains to prove that (1) implies (3).  Suppose that for all $1\le i<j<k\le n$, $\overline{\mu}_{ijk}(L)=\overline{\mu}_{ijk}(J)$.  For $i=1,\dots, n$ if $\Arf(L_i)\neq \Arf(J_i)$ then modify $L_i$ by band summing with the trefoil knot.  Call the resulting link $L^0$.  By Proposition~\ref{Arf},  $\Arf(L_i^0) = \Arf(L_i)+1=\Arf(J_i)$, since the Arf invariant takes values in $\Z/2$.  Notice that as in Figure~\ref{fig:Trefoil} each band sum with the trefoil may be obtained by first band summing with the unknot (which does not change the link type of $L$) and then performing $-1$ surgery along two curves each of which have zero linking number with every component of $L$.  

Next, for all $1\le i\le j\le n$, if $\overline\mu_{iijj}(L^0)\neq \overline\mu_{iijj}(J)$ then we band sum the components $L_i^0$ and $L_j^0$ of $L^0$ with the $m_{ij}$ twisted Whitehead link where $m_{ij}=\overline\mu_{iijj}(J) - \overline\mu_{iijj}(L^0)$.  Call the resulting link $L^1$.  By Proposition~\ref{Mu}, 
$
\overline{\mu}_{iijj}(L^1) = \overline{\mu}_{iijj}(L^0) + m_{ij} = 
\overline{\mu}_{iijj}(J).
$
This band sum can be obtained by first band summing with the 2-component unlink (which does not change the link type of  $L^0$) and then performing $\frac{-1}{m_{ij}}$ surgery along a single curve as in Figure~\ref{fig:Whitehead} which has zero linking numbers with all components of $L^0$.

By design, we now  have that for all $i,j,k$, $\overline{\mu}_{ijk}(L^1) =\overline{\mu}_{ijk}(L) = \overline{\mu}_{ijk}(J)$, $\Arf(L_i^1)=\Arf(J_i)$ and $\overline\mu_{iijj}(L^1) = \overline{\mu}_{iijj}(J)$.  According to \cite[Theorem 1]{MartinThesis},  $L^1$ and $J$ are be related by a sequence of band pass moves, depicted in Figure~\ref{fig:BPSurgery}. Thus, there exist a collection of curves each of which has zero linking number with every component of $L^1$ such that $J$ is the result of  modifying $L^1$ via $0$-surgery along these curves.  

Let $\gamma$ be the collection of all of the curves of the preceding three paragraphs.  We see that $J$ can be obtained by modifying $L$ by surgery along these curves each of which has zero linking number with every component of $L$.  This completes the proof.
\end{proof}

\bibliographystyle{plain}
\bibliography{biblio}  

\begin{thebibliography}{10}

\bibitem{Alexander28}
J.~W. Alexander.
\newblock Topological invariants of knots and links.
\newblock {\em Trans. Amer. Math. Soc.}, 30(2):275--306, 1928.

\bibitem{Cimasoni2004}
David Cimasoni.
\newblock A geometric construction of the conway potential function.
\newblock {\em Commentarii Mathematici Helvetici}, 79(1):124--146, 2004.

\bibitem{CimFlo}
David Cimasoni and Vincent Florens.
\newblock Generalized {S}eifert surfaces and signatures of colored links.
\newblock {\em Transactions of the American Math Society}, 360(3), March 2008.

\bibitem{C4}
Tim~D. Cochran.
\newblock Derivatives of links: {M}ilnor's concordance invariants and
  {M}assey's products.
\newblock {\em Mem. Amer. Math. Soc.}, 84(427):x+73, 1990.

\bibitem{Cooper82}
D.~Cooper.
\newblock The universal abelian cover of a link.
\newblock In {\em Low-dimensional topology ({B}angor, 1979)}, volume~48 of {\em
  London Math. Soc. Lecture Note Ser.}, pages 51--66. Cambridge Univ. Press,
  Cambridge-New York, 1982.

\bibitem{Kauffman81}
Louis~H. Kauffman.
\newblock The {C}onway polynomial.
\newblock {\em Topology}, 20(1):101--108, 1981.

\bibitem{OnKnots}
Louis~H. Kauffman.
\newblock {\em On knots}, volume 115 of {\em Annals of Mathematics Studies}.
\newblock Princeton University Press, Princeton, NJ, 1987.

\bibitem{L5}
J.~Levine.
\newblock Knot cobordism groups in codimension two.
\newblock {\em Comment. Math. Helv.}, 44:229--244, 1969.

\bibitem{MartinThesis}
Taylor Martin.
\newblock Classification of links up to 0-solvability.
\newblock Preprint available at http://arxiv.org/abs/1511.00156.

\bibitem{MellorMelvin2003}
Blake Mellor and Paul Melvin.
\newblock A geometric interpretation of {M}ilnor's triple linking numbers.
\newblock {\em Algebr. Geom. Topol.}, 3:557--568 (electronic), 2003.

\bibitem{M2}
John Milnor.
\newblock Isotopy of links. {A}lgebraic geometry and topology.
\newblock In {\em A symposium in honor of S. Lefschetz}, pages 280--306.
  Princeton University Press, Princeton, N. J., 1957.

\bibitem{NS03}
Swatee Naik and Theodore Stanford.
\newblock A move on diagrams that generates {$S$}-equivalence of knots.
\newblock {\em J. Knot Theory Ramifications}, 12(5):717--724, 2003.

\bibitem{Rolfsen}
Dale Rolfsen.
\newblock {\em Knots and links}, volume~7 of {\em Mathematics Lecture Series}.
\newblock Publish or Perish Inc., Houston, TX, 1990.
\newblock Corrected reprint of the 1976 original.

\bibitem{Seifert35}
H.~Seifert.
\newblock \"{U}ber das {G}eschlecht von {K}noten.
\newblock {\em Math. Ann.}, 110(1):571--592, 1935.

\bibitem{Seifert50}
H.~Seifert.
\newblock On the homology invariants of knots.
\newblock {\em Quart. J. Math., Oxford Ser. (2)}, 1:23--32, 1950.

\end{thebibliography}

\end{document}